%% file: inexact_newton_arxiv.tex
\documentclass[12pt]{article}

\usepackage{amsopn,amssymb,amsthm,amsmath,bm}
\usepackage{paralist}
\usepackage{algorithm,algorithmic}
\usepackage{color}
\usepackage{graphicx}
\usepackage{comment}
\usepackage{hyperref}
\hypersetup{
     colorlinks   = true,
     linkcolor    = blue,
     citecolor    = green
}
\usepackage{multirow}
\usepackage{multicol}
\usepackage{float}
\usepackage{subfigure}
\usepackage{enumitem}
\usepackage{booktabs}
\usepackage[style=numeric,natbib=true,backend=bibtex, sortcites=true, maxbibnames=9, maxcitenames=2]{biblatex}

\usepackage{thmtools}
\usepackage{thm-restate}
\usepackage{pifont} 
\usepackage{mathrsfs}

\usepackage[capitalise]{cleveref}

\bibliography{references}
\usepackage{framed}

\newlength{\defbaselineskip}
\usepackage[a4paper , total={8.5in, 11in}, margin=1in]{geometry}

\input{newcommands}

\begin{document}
\title{\Large Inexact Non-Convex Newton-Type Methods}
\author{
	Zhewei Yao
	\thanks{
		Department of Mathematics,
		University of California at Berkeley, 
		Email: zheweiy@berkeley.edu
	}
	\and
	Peng Xu
	\thanks{
		Institute for Computational and Mathematical Engineering,
		Stanford University,
		Email: pengxu@stanford.edu
	}
	\and
	Farbod Roosta-Khorasani
	\thanks{
		School of Mathematics and Physics, University of Queensland, Brisbane, Australia, and 
		International Computer Science Institute, Berkeley, USA,     
		Email: fred.roosta@uq.edu.au
	}
	\and
	Michael W. Mahoney
	\thanks{
	International Computer Science Institute and Department of Statistics, 
	University of California at Berkeley,    
	Email: mmahoney@stat.berkeley.edu
}
}

\date{\today}
\maketitle
	

\begin{abstract}
For solving large-scale \emph{non-convex} problems, we propose \emph{inexact} variants of trust region and adaptive cubic regularization methods, which, to increase efficiency, incorporate various approximations. In particular, in addition to \emph{approximate sub-problem solves}, both the \emph{Hessian and the gradient are suitably approximated}. Using rather mild conditions on such approximations, we show that our proposed inexact methods achieve similar \emph{optimal worst-case iteration complexities} as the exact counterparts. Our proposed algorithms, and their respective theoretical analysis, \emph{do not require knowledge of any unknowable problem-related quantities}, and hence are easily \emph{implementable} in practice. In the context of finite-sum problems, we then explore randomized sub-sampling methods as ways to construct the gradient and Hessian approximations and examine the empirical performance of our algorithms on some real datasets.
\end{abstract}

\input{introduction}
\input{related_work}
\input{notation}
\input{main}
\input{tr_thm}
\input{arc_thm}

\input{sampling}

\input{experiment}

\input{conclusion}


\paragraph{Acknowledgments.}
MM gratefully acknowledges the support of DARPA, ONR, and the NSF.
FR gratefully acknowledges the support of DARPA, the Australian Research Council through a Discovery Early Career Researcher Award (DE180100923) and the Australian Research Council Centre of Excellence for Mathematical and Statistical Frontiers (ACEMS).

\printbibliography

\newpage
\onecolumn
\appendix
\input{appendix}
\input{tr_proof}
\input{arc_opt_for_subopt}

\input{arc_opt}


\end{document}

%% file: newcommands.tex
\usepackage{multirow}

\usepackage{enumitem}
\setlist[enumerate,1]{leftmargin=*,wide=0em, label = {\bfseries \roman*.}}
\setlist[itemize,1]{leftmargin=*,wide=0em}

\usepackage{arydshln}
\newcommand\bbR{\ensuremath{\mathbb{R}}} 

\newcommand{\Abs}[1]{\left |#1\right|}
\newcommand{\norm}[1]{{\left\|#1\right\|}}

\newcommand\F{\ensuremath{\mathcal{F}}} %

\newtheorem{theorem}{Theorem}

\newtheorem{definition}{Definition}
\newtheorem{lemma}[theorem]{Lemma}

\newtheorem{corollary}[theorem]{Corollary}
\newtheorem{assumption}{Assumption}

\newtheorem{remk}[theorem]{Remark}

\newtheorem{condition}{Condition}



\newif\ifwithcomments
\withcommentstrue
\ifwithcomments
\newcommand{\xxx}[1]{{\color{red} (#1)}}
\else
\newcommand{\xxx}[1]{}
\fi

\newcommand{\red}[1]{\textcolor{red}{#1}}

\definecolor{forestgreen}{rgb}{0.13, 0.55, 0.13}
\newcommand{\forestgreen}[1]{\textcolor{forestgreen}{#1}}

\DeclareMathOperator{\bigO}{\mathcal{O}}

\DeclareMathOperator{\argmin}{argmin}

\def\lmin{\lambda_{\min}}
\newcommand{\vv}[1] {\mathbf{#1}}

\def\g{\vv{g}}

\def\u{\vv{u}}

\def\s{\vv{s}}

\def\x{\vv{x}}

\def\w{\vv{w}}

\def\F{\vv{F}}

\def\H{\vv{H}}

\def\I{\vv{I}}

\let\^=\hat
\let\-=\mathbf

\def\epsilong{\epsilon_g}
\def\epsilonh{\epsilon_H}

\def\sigmat{\sigma_t}
\def\lambdam{\lambda_{\min}}
\def\Deltat{\Delta_t}

\def\xii{\mbox{\boldmath$\xi$\unboldmath}}

\def\argmin{\mathop{\rm argmin}}

\let\<=\langle
\let\>=\rangle

%% file: introduction.tex
\section{Introduction}
\label{sec:intro}
We consider the following general optimization problem:
\begin{align}
\label{eqn:basic_problem}
\min_{\x\in\bbR^d} F(\x),
\end{align}
where $F:\bbR^d \to \bbR $ is a smooth but possibly non-convex function. Over the last few decades, many optimization algorithms have been developed to solve~\eqref{eqn:basic_problem}~\cite{bertsekas1999nonlinear, nesterov2004introductory, boyd2004convex,nocedal2006numerical}. The bulk of these efforts in the machine learning (ML) community has been on developing first-order methods, i.e., those which solely rely on gradient information. Such algorithms, however, can generally be, at best, ensured to converge to \emph{first-order stationary points}, i.e., $\x$ for which $\norm{\nabla F(\x)} = 0$, which include \emph{saddle-points}.  However, it has been argued that converging to saddle points can be undesirable for obtaining good generalization errors with many non-convex machine learning models, such as deep neural networks~\cite{dauphin2014identifying,choromanska2015loss,saxe2013exact,lecun2012efficient}. In fact, it has also been shown that in certain settings, existence of ``bad'' local minima, i.e., sub-optimal local minima with high training error, can significantly hurt the performance of the trained model at test time~\cite{swirszcz2016local,fukumizu2000local}. Important cases have also been demonstrated where, stochastic gradient descent, which is, nowadays, arguably the optimization method of choice in ML, indeed stagnates at high training error~\cite{he2016deep}. 
As a result, scalable algorithms which avoid saddle points and guarantee convergence to a local minimum are desired.

Second-order methods, on the other hand, which effectively employ the curvature information in the form of Hessian, have the potential for convergence to second-order stationary points, i.e., $\x$ for which $\norm{\nabla F(\x)} = 0$ and $\nabla^2 F(\x) \succeq 0$. However, the main challenge preventing the ubiquitous use of these methods is the computational costs involving the application of the underlying matrices, e.g., Hessian.  
%
In an effort to address these computational challenges, for large-scale convex settings, stochastic variants of Newton's methods have been shown, not only, to enjoy desirable theoretical properties, e.g., fast convergence rates and robustness to problem ill-conditioning~\cite{roosta2016sub1,xu2016sub,bollapragada2016exact}, but also to exhibit superior empirical performance~\cite{roosta2016sub2,berahas2017investigation}. 

For non-convex optimization, however, the development of similar efficient methods lags significantly behind. Indeed, designing efficient and Hessian-free variants of classic non-convex Newton-type methods such as trust-region (TR)~\cite{conn2000trust}, cubic regularization (CR) \cite{nesterov2006cubic}, and its adaptive variant (ARC)~\cite{cartis2011adaptiveI,cartis2011adaptiveII}, can be an appropriate place to start bridging this gap. 
This is, in particular, encouraging since Hessian-free methods only involve Hessian-vector products, which 
in many cases including neural networks \cite{griewank1993some,pearlmutter1994fast}, are computed as efficiently as evaluating gradients. In this light, coupling \emph{stochastic approximation} with \emph{Hessian-free} techniques indeed holds promise for many of the challenging ML problems of today e.g., \citet{martens2010deep,xuNonconvexEmpirical2017,regier2017fast}.

In many applications, however, even accessing the exact gradient information can be very expensive. For example, for finite-sum problems in high dimensions, where 
\begin{align}
\label{eqn:finite_sum_problem}
F(\x)  = \frac{1}{n} \sum_{i=1}^{n} f_{i}(\x),
\end{align}
computing the exact gradient requires a pass over the entire data, which can be costly when $ n \gg 1 $. Inexact access to \emph{both} the gradient and Hessian information can usually help reduce the underlying  computational costs \cite{roosta2016sub1,roosta2016sub2,tripuraneni2017stochasticcubic}. 

\begin{table*}[!tb]
\caption{Comparison of optimal worst-case iteration complexities for convergence to a $(\epsilon, \sqrt{\epsilon})-$ Optimality (cf.\ \cref{def:optimality}), among different second-order methods for non-convex optimization. TR and CR refer, respectively, to the class of trust region and cubic regularization methods. An algorithm is said to contain ``Knowable Parameters'' if its parameter settings do not require knowledge of any constant which can not be obtained/estimated in practice, e.g., Lipschitz continuity constants. ``Practically Implementable'' refers to an algorithm which not only does not require exhaustive search over hyper-parameter space for tuning, but also failure to precisely ``fine-tune'' is not likely to result in unwanted behaviors, e.g., divergence or stagnation.\label{tab:table1}}
\centering
\begin{tabular}{ccccc}
\toprule
\multirow{2}{1.5cm}{\centering \small{Method Class}} & \multirow{2}{1.8cm}{\centering \small{Iteration \\ Complexity}} &  \multirow{2}{1.8cm}{\centering \small{Inexact Hessian}} & \multirow{2}{1.5cm}{\centering \small{Inexact Gradient}} & \multirow{2}{5cm}{\centering \small{Knowable Parameters and/or Practically Implementable}}\\[2.5ex]
\midrule
TR \cite{cartis2012complexity} & $\bigO(\epsilon^{-2.5})$ & \forestgreen{\ding{51}} &\red{\ding{55}} & \forestgreen{\ding{51}}
\\
TR \cite{xuNonconvexTheoretical2017} & $\bigO(\epsilon^{-2.5})$ & \forestgreen{\ding{51}} & \red{\ding{55}} & \forestgreen{\ding{51}}
\\
{\bf TR (Algorithm~\ref{alg:trust_region})} & $\bigO(\epsilon^{-2.5})$ & \forestgreen{\ding{51}} & \forestgreen{\ding{51}} & \forestgreen{\ding{51}}
\\
\hdashline \\ [-2ex]
CR \cite{cartis2012complexity} & $\bigO(\epsilon^{-1.5})$ & \forestgreen{\ding{51}} &\red{\ding{55}} & \forestgreen{\ding{51}}
\\
CR \cite{xuNonconvexTheoretical2017} & $\bigO(\epsilon^{-1.5})$ & \forestgreen{\ding{51}} & \red{\ding{55}} & \forestgreen{\ding{51}}
\\
CR \cite{tripuraneni2017stochasticcubic} & $\bigO(\epsilon^{-1.5})$ & \forestgreen{\ding{51}} & \forestgreen{\ding{51}} & \red{\ding{55}}
\\
{\bf CR (Algorithm~\ref{alg:arc})} & $\bigO(\epsilon^{-1.5})$ & \forestgreen{\ding{51}} & \forestgreen{\ding{51}} & \forestgreen{\ding{51}}
\\
\bottomrule
\end{tabular}
\end{table*}

\subsection{Contributions}
Here, we further these ideas by analyzing \emph{inexact} variants of TR and ARC algorithms, which, to increase efficiency, incorporate \emph{approximations} of
\begin{itemize}[wide=0em]
	\item \emph{gradient} and \emph{Hessian information},
	\item solutions of the underlying \emph{sub-problems}.
\end{itemize}
Our algorithms are motivated by the works of \citet{cartis2012complexity,xuNonconvexTheoretical2017}, which analyzed the variants of TR and ARC where the Hessian is approximated but accurate gradient information is required. We will show that, under mild conditions on approximations of the gradient, Hessian, as well as subproblem solves, our proposed inexact TR and ARC algorithms can retain the same optimal worst-case convergence guarantees as the exact counterparts \citep{cartis2011optimal,cartis2012complexity}. More specifically, to achieve $(\epsilon_g, \epsilon_H)$-Optimality (cf.\ \cref{def:optimality}), we show the following.
\begin{itemize}[wide=0em]
	\item \emph{Inexact TR (\cref{alg:trust_region})}, under \cref{cond:tr_gh} on the gradient and Hessian approximation, and \cref{cond:tr_approx_sln} on approximate sub-problem solves, requires the optimal iteration complexity of $\bigO(\max\{\epsilon_g^{-2}\epsilon_H^{-1},\epsilon_H^{-3}\})$. Please see \cref{subsec:tr_thm} for more details.
	\item \emph{Inexact ARC (\cref{alg:arc})}, under \cref{cond:arc_gh} on the gradient and Hessian approximation, and \cref{cond:arc_sub_appr} on approximate sub-problem solves, requires less than $\bigO(\max\{\epsilon_g^{-2},\epsilon_H^{-3}\})$, which is sub-optimal. These two conditions are given below in \cref{subsec:arc_subopt}.
     However, under respectively stronger Conditions \ref{cond:arc_opt_gh} and \ref{cond:arc_optimal2}, the optimal iteration complexity of $\bigO(\max\{\epsilon_g^{-3/2}, \epsilon_H^{-3}\})$ is recovered. The details are shown in \cref{subsec:arc_opt} 
\end{itemize}


An important aspect of our contribution is that our proposed algorithms, and their respective analysis, do not assume knowledge of any unknowable problem-related quantities, e.g., Lipschitz continuity constants of the gradient and the Hessian, which cannot be obtained in practice. 
Making such assumptions often helps with carrying out the theoretical analysis, but it has the unwanted practical consequence that the resulting algorithms are practically hard to implement, if possible at all. 
For example, one solution to parameterizing algorithms in terms of unknowable quantities is to introduce hyper-parameters and then resort to expensive/exhaustive hyper-parameter tuning in order to achieve desirable performance.
On the contrary, as part of our contributions, we propose theoretically optimal algorithms whose \emph{implementations require no knowledge of unknowable and/or problem-related quantities}.  

In addition to our theoretical contributions, we empirically demonstrate the advantages of our methods on several real datasets; see \cref{sec:exp} for more details.
In addition to showing good performance, e.g., in terms of efficiency, we also highlight some additional features of our algorithms such as robustness to hyper-parameter tuning.
This is a great practical advantage.
In particular, in \cref{fig:robustness_comparision}, we show our Inexact ARC (\cref{alg:arc}) is insensitive w.r.t. the cubic regularization parameter.
However, for a related algorithm based on unknowable problem-related quantities, the performance is highly dependent on the choice of its hyper-parameter.

A snapshot of comparison among our proposed methods and other similar algorithms is given in Table~\ref{tab:table1}.

%% file: related_work.tex
\subsection{Related work}
Due to the resurgence of deep learning, recently, there has been a rise of interest in efficient non-convex optimization algorithms. For non-convex problems, where saddle points have been shown to give understandable generalization performance, several first-order methods, especially variants of stochastic gradient descent(SGD), have been devised that promise to efficiently escape saddle points and, instead, converge to second-order stationary point~\cite{ge2015escaping,jin2017escape,levy2016power}. 

As for second-order methods, there have been a few empirical studies of the application of inexact curvature information for, mostly, deep-learning applications, e.g., see the pioneering work of~\cite{martens2010deep} and follow-ups~\cite{wiesler2013investigations,vinyals2012krylov,he2016large,kiros2013training}. However, the theoretical understanding of these inexact methods remains largely under-studied. 
Among a few related theoretical prior works, most notably are the ones which study derivative-free and probabilistic models in general, and Hessian approximation in particular for trust-region methods~\cite{conn2009global, chen2015stochastic,blanchet2016convergence,bandeira2014convergence,larson2016stochastic,shashaani2016astro, gratton2017complexity}. 

For cubic regularization, the seminal works of~\citet{cartis2011adaptiveI,cartis2011adaptiveII} are the first to study Hessian approximation and the resulting algorithm is an adaptive variant of the cubic regularization, referred to as ARC. In~\citet{cartis2012complexity}, similar Hessian inexactness is also extended to trust region methods. However, to guarantee optimal complexity, 
they require not only exact gradient information but also progressively accurate Hessian information which can be difficulty to satisfy.
For minimization of a finite-sum~\eqref{eqn:finite_sum_problem}, a sub-sampled variant of ARC was proposed in \citet{kohler2017subsampledcubic}, which directly rely on the analysis of~\citet{cartis2011adaptiveI,cartis2011adaptiveII}. 
More recently, \citet{tripuraneni2017stochasticcubic} proposed a stochastic variant of cubic regularization, henceforth referred to as SCR, in which, in order to guarantee optimal performance, only stochastic gradient and Hessian is required. However their algorithm and analysis rely on assuming, rather unknowable, problem related constants, e.g, Lipschitz continuity of the gradient and Hessian. 

In the context of both TR and ARC, under milder Hessian approximation conditions than prior works, \citet{xuNonconvexTheoretical2017} recently analyzed optimal complexity of variants in which the Hessian is approximated, but the exact gradient is used. Our approach here builds upon the ideas in \citet{xuNonconvexTheoretical2017}. 

%% file: notation.tex
\section{Notation and Assumptions}

Unlike convex problems, where tracking the first-order condition, i.e., norm of the gradient, is sufficient to evaluate (approximate) optimality, in non-convex settings, the situation is much more involved, e.g., see examples of~\citet{murty1987some,hillar2013most}. In this light, one typically sets out to design algorithms that can guarantee convergence to approximate second-order optimality. 
\begin{definition}[$(\epsilong,\epsilonh)$-Optimality]\label{def:optimality}
Given $0<\epsilong,\epsilonh<1$, $\x$ is an $(\epsilong,\epsilonh)$-Optimal solution of \eqref{eqn:basic_problem}, if \footnote{Throughout the paper, $\|\cdot\|$ is $\ell$-2 norm by default. $\lmin(\cdot)$ is the minimum eigenvalue.}
\begin{align}\label{eqn:def_optimality}
	\|\nabla F(\x)\| \leq \epsilong, \quad \text{and} \quad \lambdam( \nabla^2 F(\x) )  \geq -\epsilonh.
\end{align}
\end{definition}

For our analysis throughout the paper, we make the following standard assumptions on the smoothness of objective function $F$. 
Note that, for our algorithms we do not require the actual knowledge of the following constants.

\begin{assumption}[Hessian Regularity]\label{ass:regularity} $F(\x)$ is twice differentiable. Furthermore, there are some constants $0 < L_F, K_F < \infty$ such that for any $\x = \x_t+\tau\s_t, \; \tau\in[0,1]$, we have
\begin{subequations}
	\label{eq:Hessian_regularity_F}
	\begin{align}
	&\norm{\nabla^2 F(\x) -\nabla^2 F(\x_{t})} \le L_{F} \norm{\x - \x_{t}}, \; \label{eq:Hessian_Lipschitz_F}
	\\
	& \norm{\nabla^2 F(\x_{t})} \le K_F, 	\label{eq:Hessian_boundedness_F}
	\end{align}
	\end{subequations}
where $\x_t$ and $\s_t$ are, respectively, the iterate and update direction at step $t$. 
\end{assumption}

For our inexact algorithms, we require that the approximate gradient, $\g_t$, and the inexact Hessian, $\H_t$, at each iteration $ t $, satisfy the following, rather mild, conditions.
\begin{assumption}[Gradient and Hessian Approximation Error]\label{ass:approximation}
For some $0<\delta_g, \delta_H<1$, the approximations of the gradient and Hessian at step $t$ satisfy,
\begin{align*}
\|\g_t-\nabla F(\x_t)\| &\leq \delta_g,
\\
\|\H_t-\nabla^2 F(\x_t)\| &\leq \delta_H.
\end{align*}
Note that Assumptions~\ref{ass:regularity} and~\ref{ass:approximation} imply that $\norm{\H_t} \le K_H$, where $ K_H \le K_F + \delta_H $.
\end{assumption}

%% file: main.tex
\section{Main Results}\label{sec:main}
In this section we will present our main algorithms as well as their respective analysis, i.e., inexact variants of TR (Algorithm~\ref{alg:trust_region}) and ARC (Algorithm~\ref{alg:arc}) where the gradient, Hessian and the solution to sub-problems are all approximated. All the proofs are relegated to the supplementary materials.

As it can be seen from Algorithms~\ref{alg:trust_region} and~\ref{alg:arc}, compared with the standard classical counterparts, the main differences in iterations lie in using the approximations of the gradient, the Hessian, and the solution to the corresponding sub-problem \eqref{eqn:subproblem_trust} and \eqref{eqn:subproblem_cubic}. Another notable difference is when the gradient estimate is small, i.e., $\norm{\g_t} \le \epsilon_g$, in which case our algorithm completely ignores the gradient; see Step 8 of Algorithms~\ref{alg:trust_region} and~\ref{alg:arc}. This turns out to be crucial in obtaining the optimal iteration complexity for Algorithms~\ref{alg:trust_region} and~\ref{alg:arc}; see the supplementary materials. However, in our experiments, we never needed to enforce this step and opted to retain the gradient term even when it was small.
 
\begin{remk}[Bird's-eye View of the Challenges in the Theoretical Analysis]
\label{remk:proofs}
Gradient and Hessian approximation coupled with not employing any problem related-constants in our algorithms indeed further complicates the analysis. For example, approximating the gradient and Hessian introduces error terms throughout the analysis that are of different orders of magnitude. Controlling such drastically different error growths involves additional complications.  
Furthermore, by not incorporating unknowable problem-related constants, e.g. $ L_{F}, K_{F} $, in our algorithms, many relations in our analysis, e.g., discrepancy between the decrease suggested by the sub-problems, i.e., \eqref{eqn:subproblem_trust} and \eqref{eqn:subproblem_cubic}, and what is actually obtained in the objective, i.e., $\F(\x_t+\s)-\F(\x_t)$, had to be established indirectly. 
(Assuming knowledge of these constants makes the theory much easier, but it has the serious drawback of introducing additional hyper-parameters, the values of which must be determined.)
Details are given in the supplementary materials.
\end{remk}

\begin{figure*}[ttt!]
\begin{minipage}[t]{.45\textwidth}
\begin{algorithm}[H]
	\caption{Inexact TR}
	\label{alg:trust_region}
	\begin{algorithmic}[1]
		\STATE {\bf Input:} 
		\begin{compactitem}[\bfseries -]
			\item Starting point: $\x_0$
			\item Initial trust-region radius: $0 < \Delta_{0}  < \infty$
			\item Other Parameters: $\epsilon_{g}, \epsilon_{H}$, $0<\eta \le 1, \gamma> 1$. 
		\end{compactitem}
		\FOR{$ t = 0,1,\ldots $}
		\STATE Set the approximate gradient $\g_t$ and Hessian $\H_t$,
		\IF{$\norm{\g_{t}} \le \epsilon_g, ~\lmin(\H_t) \ge -\epsilon_H\;$}  
		\STATE  Return $\x_t$
		\ENDIF
		\IF{$\|\g_{t}\| \le \epsilon_g$}
		\STATE $\g_t=0$ 
		\ENDIF 
		\STATE $ \s_t \approx \argmin_{\|\s\|\leq\Deltat} \; \<\g_t, \s\> + \frac12 \langle\s,\H_t\s\rangle $
		\STATE Set $\rho_t \triangleq \dfrac{F(\x_t) - F(\x_t + \s_t)}{-m_t(\s_t)}$
		\IF {$\rho_t \ge \eta$}
		\STATE $\x_{t+1} = \x_t + \s_t$ \;and\; $\Delta_{t+1} = \gamma \Delta_t$
		\ELSE
		\STATE $\x_{t+1} = \x_t$ \;and\; $\Delta_{t+1} = \Delta_t/\gamma$
		\ENDIF
		\ENDFOR
		\STATE {\bf Output:} $\x_{t}$
	\end{algorithmic}
\end{algorithm}
\end{minipage}%
\hfill
\begin{minipage}[t]{.45\textwidth}
\begin{algorithm}[H]
    \caption{Inexact ARC}
	\label{alg:arc}
	\begin{algorithmic}[1]
		\STATE {\bf Input:} 
		\begin{compactitem}[\bfseries -]
			\item Starting point: $\x_0$
			\item Initial regularization parameter: $0 < \sigma_{0}  < \infty$ 
			\item Other Parameters: $\epsilon_{g}, \epsilon_{H}$, $0<\eta\le 1, \gamma > 1$.
		\end{compactitem}
		\FOR{$ t = 0,1,\ldots $}
		\STATE Set the approximate gradient $\g_t$ and Hessian $\H_t$, 
		\IF{$\norm{\g_{t}} \le \epsilon_g, ~\lmin(\H_t) \ge -\epsilon_H\;$}  
		\STATE  Return $\x_t$
		\ENDIF
		\IF{$\|\g_{t}\| \le \epsilon_g$} 
		\STATE $\g_t=0$
		\ENDIF 
		\STATE $ \s_t \approx \argmin_{\s\in\bbR^d} \; \<\g_t, \s\> + \frac12\<\s, \H_t\s\> + \frac{\sigma_t}{3}\|\s\|^3 $
		\STATE Set $\rho_t \triangleq \dfrac{F(\x_t) - F(\x_t + \s_t)}{-m_t(\s_t)}$ 
		\IF {$\rho_t \ge \eta$}
		\STATE $\x_{t+1} = \x_t + \s_t$ \;and\; $\sigma_{t+1} = \sigma_t/\gamma$
		\ELSE
		\STATE $\x_{t+1} = \x_t$ \;and\; $\sigma_{t+1} = \gamma\sigma_t$
		\ENDIF
		\ENDFOR
		\STATE {\bf Output:} $\x_{t}$
	\end{algorithmic}
\end{algorithm}
\end{minipage}%
\end{figure*}

%% file: tr_thm.tex
\subsection{Inexact Trust Region: \cref{alg:trust_region}}\label{subsec:tr_thm}
The inexact TR algorithm is depicted in \cref{alg:trust_region}. Every iteration of \cref{alg:trust_region} involves approximate solution to a sub-problem of the form
\begin{subequations}
\label{eqn:subproblem_trust}
\begin{align}
\s_t \approx \argmin_{\|\s\|\leq\Delta_t} m_t(\s),
\end{align}
where
\begin{align}
m_t(\s) =\begin{cases} 
\<\g_t, \s\> + \frac12\<\s, \H_t\s\>,&\|\g_t\| \geq \epsilong\\
\<\s, \H_t\s\>, &\text{Otherwise}
\end{cases}.
\end{align}
\end{subequations}

\begin{subequations}
Classically, the analysis of TR method involves obtaining a minimum descent along two important directions, namely negative gradient and (approximate) negative curvature. Updating the current point using these directions gives, respectively, what are known as Cauchy Point and Eigen Point~\cite{conn2000trust}. In other words, Cauchy Point and Eigen Point, respectively, correspond to the optimal solution of \eqref{eqn:subproblem_trust} along the negative gradient and the negative curvature direction (if it exists). 
\begin{definition}[Cauchy Point for \cref{alg:trust_region}]\label{def:cauchy_tr}
When $\norm{\g_t} \ge \epsilon_g$, Cauchy Point for \cref{alg:trust_region} is obtained from \eqref{eqn:subproblem_trust} as 
\begin{align}
\s_t^C = -\alpha^C  \frac{\g_t}{\norm{\g_t}},~~\alpha^C = \argmin_{0\le\alpha\leq\Delta_t} m_t(-\alpha \frac{\g_t}{\norm{\g_t}}).
\end{align}
\end{definition}
\begin{definition}[Eigen Point for \cref{alg:trust_region}]\label{def:eigen_tr}
When $\lmin(\H_t) \le -\epsilon_H$, Eigen Point for \cref{alg:trust_region} is obtained from \eqref{eqn:subproblem_trust} as  
\begin{align}
\s_t^E = \alpha^E\u_t,~~\alpha^E = \argmin_{|\alpha|\leq\Delta_t} m_t(\alpha \u_t),
\end{align}
where $\u_t$ is an approximation to the corresponding negative curvature direction, i.e., for some $ 0 < \nu < 1 $,
\begin{align*}
\<\u_t, \H_t\u_t\> \le \nu\lmin(\H_t) \; \text{ and } \; \norm{\u_t} = 1.
\end{align*}
\end{definition}
\end{subequations}
The properties of Cauchy and Eigen Points have been studied in \citet{cartis2011adaptiveI,cartis2011adaptiveII,xuNonconvexTheoretical2017}, and are also stated in Lemmas \ref{lemma:cauchy} and \ref{lemma:eig} in the supplementary materials.

We are now ready to give the convergence guarantee of \cref{alg:trust_region}. For this, we first present sufficient conditions (\cref{cond:tr_gh}) on the degree of inexactness of the gradient and Hessian. In other words, we now give conditions on $ \delta_g, \delta_H $ in \cref{ass:approximation} which ensure convergence.

\begin{condition}[Gradient and Hessian Approximation for \cref{alg:trust_region}]\label{cond:tr_gh}
Given the termination criteria $\epsilon_g, \epsilon_H$ in \cref{alg:trust_region}, we require the inexact gradient and Hessian to satisfy
\begin{align}
\delta_g  \le \frac{(1 - \eta)\epsilong}{4} \quad \text{and} \quad  \delta_H \leq \min\left\{\frac{(1 -\eta)\nu\epsilon_H}{2},1\right\}. \label{eq:tr_gh}
\end{align}
\end{condition}

\cref{cond:tr_gh} imposes approximation requirements on the inexact gradient and Hessian. More specifically, \eqref{eq:tr_gh} implies that we must seek to have $\delta_g \in \bigO(\epsilon_g), \delta_H \in \bigO(\epsilon_H)$. These bounds are indeed the minimum requirements for the gradient and Hessian approximations to achieve $(\epsilon_g, \epsilon_H)$-Optimality; see the termination step for \cref{alg:trust_region}.

In \cref{alg:trust_region}, sub-problem \eqref{eqn:subproblem_trust} need only be solved approximately. Indeed, in large-scale settings, obtaining the exact solution of the sub-problem \eqref{eqn:subproblem_trust} is computationally prohibitive. For this, as it has been classically done, we require that an approximate solution of the sub-problem satisfies what are known as Cauchy and Eigen Conditions~\cite{conn2000trust,cartis2010complexity,xuNonconvexTheoretical2017}. In other words, we require that an approximate solution to \eqref{eqn:subproblem_trust} is at least as good as Cauchy and Eigen points in Definitions~\ref{def:cauchy_tr} and~\ref{def:eigen_tr}, respectively. \cref{cond:tr_approx_sln}  makes this explicit.

\begin{condition}[Approximate solution of \eqref{eqn:subproblem_trust} for \cref{alg:trust_region}]\label{cond:tr_approx_sln}
We require to solve the sub-problem (\ref{eqn:subproblem_trust}) approximately to find $\s_t$ such that
\begin{align*}
m_t(\s_t) \leq m_t(\s_t^C),~~m_t(\s_t) \leq m_t(\s_t^E), 
\end{align*}
where $\s_t^C$ and $\s_t^E$ are Cauchy and Eigen points, as in Definitions~\ref{def:cauchy_tr} and~\ref{def:eigen_tr}, respectively.
\end{condition}

It is not hard to see that if \eqref{eqn:subproblem_trust} is solved restricted to any sub-space containing $\text{Span}\{\s_t^C, \s_t^E\}$, the corresponding optimal solution  satisfies \cref{cond:tr_approx_sln}. 

Under Assumptions \ref{ass:regularity} and \ref{ass:approximation} , as well as Conditions \ref{cond:tr_gh} and \ref{cond:tr_approx_sln}, we are now ready to give the optimal iteration complexity of \cref{alg:trust_region} as stated in \cref{thm:tr_main}. 

\begin{restatable}[Optimal Complexity of \cref{alg:trust_region}]{theorem}{thmtrmain}\label{thm:tr_main}
Let \cref{ass:regularity} hold 
and suppose that  $\g_t$ and $\H_t$ satisfy \cref{ass:approximation} with $ \delta_g $ and $ \delta_H  $ under \cref{cond:tr_gh}. If the approximate solution to the sub-problem~\eqref{eqn:subproblem_trust} satisfies \cref{cond:tr_approx_sln}, then \cref{alg:trust_region} terminates after at most
\begin{align*}
T \in \mathcal{O}\left(\max\left\{\epsilong^{-2}\epsilonh^{-1},\epsilonh^{-3}\right\}\right),
\end{align*}
iterations. 
\end{restatable}

The worst iteration complexity of \cref{thm:tr_main} matches the bound obtained in \citet{conn2000trust,cartis2012complexity,xuNonconvexTheoretical2017}, which is known to be optimal in worst-case sense~\citep{cartis2012complexity}. Further, it follows immediately  that the terminating points of \cref{alg:trust_region} satisfies $\|\g_T\|\leq \epsilong + \delta_g$ and $\lmin(\H_T)\geq -\epsilonh-\delta_h$, i.e. $\x_T$ is a $(\epsilong+\delta_g,\epsilonh+\delta_h)$-optimal solution of \eqref{eqn:basic_problem}.

%% file: arc_thm.tex
\subsection{Inexact ARC: \cref{alg:arc}}\label{subsec:arc_thm}
The inexact ARC algorithm is given in \cref{alg:arc}. Every iteration of \cref{alg:arc} involves an approximate solution to the following sub-problem: 
\begin{subequations}
	\label{eqn:subproblem_cubic}
	\begin{align}
	\s_t \approx \argmin_{\s \in \mathbb{R}^{d}}  m_t(\s),
	\end{align}
	where
	\begin{equation}
	m_t(\s)=\begin{cases} 
	\<\g_t, \s\> + \frac12\<\s, \H_t\s\> + \frac{\sigma_t}{3}\|\s\|^3,&\|\g_t\|\geq\epsilong\\
	\<\s, \H_t\s\> + \frac{2\sigma_t}{3}\|\s\|^3, &\text{Otherwise}
	\end{cases}.
	\end{equation}
\end{subequations}

\begin{subequations}
Similar to Section~\ref{subsec:tr_thm}, our analysis for inexact ARC also involves Cauchy and Eigen points obtained from \eqref{eqn:subproblem_cubic} as follows.
\begin{definition}[Cauchy Point for \cref{alg:arc}]\label{def:cauchy_arc}
When $\norm{\g_t} \ge \epsilon_g$, Cauchy Point for \cref{alg:arc} is obtained from \eqref{eqn:subproblem_cubic} as 
\begin{equation}
\s_t^C = -\alpha^C\g_t,~~\alpha^C = \argmin_{\alpha \geq 0} m_t(-\alpha \g_t).
\end{equation}
\end{definition}
\begin{definition}[Eigen Point for \cref{alg:arc}]\label{def:eigen_arc}
When $\lmin(\H_t) \le -\epsilon_H$, Eigen Point for \cref{alg:arc} is obtained from \eqref{eqn:subproblem_cubic} as 
\begin{equation}
\s_t^E = \alpha^E\u_t,~~\alpha^E = \argmin_{\alpha \in \mathbb{R}} m_t(\alpha \u_t),
\end{equation}
where $\u_t$ is an approximation to the corresponding negative curvature direction, i.e., for some $ 0 < \nu < 1 $,
\begin{align*}
\<\u_t, \H_t\u_t\> \le \nu\lmin(\H_t) \; \text{ and } \; \norm{\u_t} = 1.
\end{align*}
\end{definition}
\end{subequations}
The properties of Cauchy Point and Eigen Point for the cubic problem can be found in \cref{lemma:arc_cauchy_lemma} and \cref{lemma:arc_eigen_lemma} in \cref{subsec:arc_proof} in the supplementary materials.

As we shall show, the worst-case iteration complexity of inexact ARC depends on how accurately we approximate the gradient and Hessian, as well as the problem solves. In \cref{subsec:arc_subopt}, we show that under \emph{nearly} minimum requirement of the gradient and Hessian approximation (\cref{cond:arc_gh}), the inexact ARC can achieve \emph{sub-optimal} complexity $\bigO(\max\{\epsilon_g^{-2}, \epsilon_H^{-3}\})$. In  \cref{subsec:arc_opt}, we then show that under more restrict approximation condition (\cref{cond:arc_opt_gh}), the \emph{optimal} worst-case complexity $\bigO(\max\{\epsilon_g^{-1.5},\epsilon_H^{-3}\})$ can be recovered.

\subsubsection{Sub-optimal Complexity for \cref{alg:arc}}
\label{subsec:arc_subopt}
In this section, we provide sufficient conditions on approximating the gradient and Hessian, as well as the subproblem solves for inexact ARC to achieve the sub-optimal complexity $\bigO(\max\{\epsilon_g^{-2}, \epsilon_H^{-3}\})$.

First, similar to Section~\ref{subsec:tr_thm}, we require that the estimates of the gradient and the Hessian satisfy the following condition.
\begin{condition}[Gradient and Hessian Approximation for \cref{alg:arc}]
\label{cond:arc_gh}
Given the termination criteria $\epsilon_g, \epsilon_H$ in \cref{alg:arc}, we require the inexact gradient and Hessian to satisfy
\begin{equation}
\label{eq:arc_gh}
\delta_g \le \frac{1-\eta}{12}\epsilon_g, ~~ \delta_H \le\frac{1-\eta}{6}\min\{\nu\epsilon_H, \sqrt{2L_F\epsilon_g}\}.
\end{equation}
\end{condition}
It is easy to see that $\delta_{g} \in \bigO(\epsilong), \; \delta_{H} \in \bigO\left(\min\left\{\sqrt{\epsilong},\epsilonh\right\}\right)$. Similar constraints on $\delta_H$ have appeared in several previous works, e.g. \citet{tripuraneni2017stochasticcubic,xuNonconvexTheoretical2017}. These are nearly minimum requirement for the approximation. In the case when $\epsilon_H = \bigO(\sqrt{\epsilon_{g}})$, \cref{cond:arc_gh} is indeed the minimum requirement.

As for solving the subproblem, we require the following.
\begin{condition}[Approximate solution of \eqref{eqn:subproblem_cubic} for \cref{alg:arc}]\label{cond:arc_sub_appr}
We require to solve the sub-problem (\ref{eqn:subproblem_cubic}) approximately such that
\begin{itemize}[leftmargin=*]
\item If $\norm{\g_t} \ge \epsilon_g$, then we take the Cauchy Point, i.e. $\s_t = \s_t^C$.
\item Otherwise, take any $\s_t$ s.t.
\begin{align*}
& m_t(\s_t) \le m_t(\s_t^E),
\\
&\<\g_t, \s_t\>+ \<\s_t, \H_t\s_t\> + \sigma_t \norm{\s_t}^3 = 0, \quad \<\g_t,\s_t\> \le 0,
\end{align*}
\end{itemize}
where $\s_t^C$ and $\s_t^E$ are Cauchy and Eigen points, as in Definitions~\ref{def:cauchy_arc} and~\ref{def:eigen_arc}, respectively.
\end{condition}
\cref{cond:arc_sub_appr} implies that when the gradient is large-enough, we take the Cauchy step. Otherwise, we update along a step which is at least, as good as the Eigen Point.


Under Assumptions \ref{ass:regularity} and \ref{ass:approximation} , as well as Conditions \ref{cond:arc_gh} and \ref{cond:arc_sub_appr}, we now present the sub-optimal complexity of \cref{alg:arc} as stated in \cref{thm:arc_main}. 
%

\begin{restatable}[Complexity of \cref{alg:arc}]{theorem}{thmarcmain}\label{thm:arc_main} 

Let \cref{ass:regularity} hold and consider any $0 < \epsilon_g,\epsilon_H < 1$. Further, suppose that  $\g_t$ and $\H_t$ satisfy \cref{ass:approximation} with $ \delta_g $ and $ \delta_H  $ under \cref{cond:arc_gh}. If the approximate solution to the sub-problem~\eqref{eqn:subproblem_cubic} satisfies \cref{cond:arc_sub_appr}, then \cref{alg:arc} terminates after at most
\begin{align*}
T \in \mathcal{O}\left(\max\left\{\epsilong^{-2},\epsilonh^{-3}\right\}\right),
\end{align*}
iterations.
\end{restatable}
\begin{remk}
To obtain similar sub-optimal iteration complexity, the sufficient condition on approximating the Hessian in \citet{xuNonconvexTheoretical2017} requires that $ \delta_{H} \in \bigO\left(\min\left\{\epsilong,\epsilonh\right\}\right) $, which is stronger than \cref{cond:arc_gh}.
\end{remk}

\subsubsection{Optimal Complexity for \cref{alg:arc}}
\label{subsec:arc_opt}
In this section, we show that by better approximation of the gradient, Hessian as well as the sub-problem \eqref{eqn:subproblem_cubic}, \cref{alg:arc} indeed enjoys the optimal iteration complexity. 

First we require the following condition on approximating the gradient and Hessian:
\begin{condition}[Gradient and Hessian Approximation for \cref{alg:arc}]
\label{cond:arc_opt_gh}
Given the termination criteria $\epsilon_g, \epsilon_H$ in \cref{alg:arc}, we require the inexact gradient and Hessian to satisfy
\begin{subequations}
\label{eq:arc_opt_gh}
\begin{align}
\delta_g & \le \frac{(1 - \eta)}{192 L_{F}} \left(\sqrt{K_H^2 + 8L_{F}\epsilong} - K_H\right)^2, \label{eq:arc_g}
\\
\delta_H &\leq \frac{(1 -\eta)}{6} \min\left\{\frac{1}{4} \left(\sqrt{K_H^2+8L_{F}\epsilong} - K_H\right), \nu\epsilon_H\right\}, \label{eq:arc_h}
\\
\delta_g &\le \delta_H \le \frac{1}{5}\epsilong. \label{eq:arc_opt2}
\end{align}
\end{subequations}
\end{condition}
\cref{cond:arc_opt_gh} implies $\delta_g = \bigO(\epsilon_g^2)$ and $\delta_H = \bigO(\min\{\epsilon_g, \epsilon_H\})$, which is strictly stronger than \cref{cond:arc_gh} in \cref{subsec:arc_subopt}. Admittedly, although \cref{cond:arc_opt_gh} allows one to obtain optimal iteration complexity of \cref{alg:arc}, it also implies more computations, e.g., for finite-sum problems of  \cref{subsec:finite-sum}, this translates to larger sampling complexities. 
We suspect that, instead of being an inherent property of \cref{alg:arc}, this is merely a by-product of our analysis. In this light, we conjecture that the same requirement as \eqref{eq:arc_gh} should also be sufficient for \cref{alg:arc}; investigating this conjecture is left for future work. 

Now we provide a sufficient condition on approximating the solution of the sub-problem \eqref{eqn:subproblem_cubic}. Here we require that the sub-problem \eqref{eqn:subproblem_cubic} is solved more accurately than in \cref{cond:arc_sub_appr}. To obtain optimal complexity, similar conditions have been considered in several previous works \citep{cartis2010complexity,xuNonconvexTheoretical2017}. Specifically we require the solution is, not only, as good as Cauchy and Eigen points, but also it satisfies an extra requirement, \eqref{eq:arc_opt1}, which accelerates the convergence to first-order critical points. 
\begin{condition}[Approximate solution of \eqref{eqn:subproblem_cubic} for \cref{alg:arc}]
\label{cond:arc_optimal2} 
Assume that we solve the sub-problem (\ref{eqn:subproblem_cubic}) with $\|\g_t\|\geq\epsilong$ approximately to find $\s_t$, such that
\begin{subequations}
\begin{align}
&m_t(\s_t) \le m_t(\s_t^C), m_t(\s_t) \le m_t(\s_t^E),
\\
&\<\g_t, \s_t\>+ \<\s_t, \H_t\s_t\> + \sigma_t \norm{\s_t}^3 = 0, \quad \<\g_t,\s_t\> \le 0,
\\
&\|\nabla m(\s_t)\| \leq \theta_t \|\g_t\|,~~~~\theta_t\leq \min\{1,\|\s_t\|\}/5 \label{eq:arc_opt1},
\end{align}
\end{subequations}
where $\s_t^C$ and $\s_t^E$ are Cauchy and Eigen points, as in Definitions~\ref{def:cauchy_arc} and~\ref{def:eigen_arc}, respectively.
\end{condition}

Under Assumptions \ref{ass:regularity} and \ref{ass:approximation} , as well as Conditions \ref{cond:arc_opt_gh} and \ref{cond:arc_optimal2}, we now present the optimal complexity of \cref{alg:arc} as stated in \cref{thm:arc_main_optimal2}. 

\begin{restatable}[Optimal Complexity of \cref{alg:arc}]{theorem}{thmarcmainoptII}\label{thm:arc_main_optimal2} Let \cref{ass:regularity} hold and consider any $0 < \epsilon_g,\epsilon_H < 1$. Further, suppose that  $\g_t$ and $\H_t$ satisfy \cref{ass:approximation} with $ \delta_g $ and $ \delta_H  $ under \cref{cond:arc_opt_gh}. If the approximate solution to the sub-problem~\eqref{eqn:subproblem_cubic} satisfies \cref{cond:arc_optimal2}, then \cref{alg:arc} terminates after at most
\begin{align*}
T \in \mathcal{O}(\max\{\epsilong^{-1.5},\epsilonh^{-3}\}),
\end{align*}
iterations.
\end{restatable}

%% file: sampling.tex
\subsection{Finite-Sum Problems}\label{subsec:finite-sum}
As a special class of \eqref{eqn:basic_problem}, we now consider non-convex finite-sum minimization of \eqref{eqn:finite_sum_problem}, where each $f_i:\bbR^d \to \bbR$ is smooth and non-convex. In big-data regimes where $n \gg 1$, one can consider sub-sampling schemes to speed up various aspects of many Newton-type methods, e.g., see \citet{roosta2016sub1,roosta2016sub2,xu2016sub,bollapragada2016exact} for such techniques in the context of convex optimization. More specifically, we consider the sub-sampled gradient and Hessian as
\begin{align}\label{eq:sub_gh}
\g \triangleq \frac{1}{\Abs{\mathcal S_g}} \sum_{i\in\mathcal S_g} \nabla f_i(\x) \;\; \text{and} \;\; \H \triangleq \frac{1}{\Abs{\mathcal S_H}} \sum_{i\in\mathcal S_H} \nabla^2 f_i(\x),
\end{align}
where $\mathcal S_g, \mathcal S_H\subset\{1,\cdots, n\}$ are the sub-sample batches for the estimates of the gradient and Hessian, respectively. In this setting, a relevant question is that of ``how large sample sizes $ \mathcal S_g$ and $\mathcal S_H $ should be to guarantee, at least with high probability, that $ \g  $ and $ \H  $ in \eqref{eq:sub_gh} satisfy \cref{ass:approximation}''. 

If sampling is done uniformly at random, we have the following sampling complexity bounds, whose proofs can be found in~\citet{roosta2016sub1,xuNonconvexTheoretical2017}. For more sophisticated sampling/sketching schemes, see \citet{pilanci2015newton,xu2016sub,xuNonconvexTheoretical2017}.
\begin{restatable}[Sampling Complexity \citep{roosta2016sub1,xuNonconvexTheoretical2017}]{lemma}{lemmasampling}
\label{lemma:sampling}
For any $0< \delta_g,\delta_H,\delta <1$, let $\g$ and $\H$ be as in \eqref{eq:sub_gh} with
\begin{align*}
\Abs{\mathcal S_g} \ge \frac{16K_g^2}{\delta_g^2} \log\frac{1}{\delta} \quad \text{and} \quad \Abs{\mathcal S_H} \ge \frac{16K_H^2}{\delta_H^2} \log\frac{2d}{\delta},
\end{align*}
where $ 0< K_g,K_H < \infty $ are such that $ \norm{\nabla f_i(\x)} \le K_g$ and $\norm{\nabla^2 f_i(\x)} \le K_H$. Then, with probability at least $1 - \delta$, \cref{ass:approximation} holds with the corresponding $\delta_g$ and $\delta_H$.
\end{restatable}

Combining \cref{lemma:sampling} with the sufficient conditions presented earlier, i.e., \cref{cond:tr_gh} for \cref{alg:trust_region} and Conditions \ref{cond:arc_gh} or \ref{cond:arc_opt_gh} for \cref{alg:arc}, we can immediately obtain, similar, but probabilistic, iteration complexities as in Sections~\ref{subsec:tr_thm}  and~\ref{subsec:arc_thm}; hence we omit the details.

%% file: experiment.tex
\section{Experiments}
\label{sec:exp}
In this section, we provide empirical results evaluating the performance of Algorithms~\ref{alg:trust_region} and~\ref{alg:arc}.
We aim to demonstrate two things: (a) that approximate gradient, approximate Hessian and approximate sub-problem solves indeed help improve the computational efficiency; and (b) that our algorithms are easy to implement and do not require expensive hyper-parameter tuning. 
We do this in the context of simple, yet illustrative, nonlinear least squares arising from the task of binary classification with squared loss\footnote{Since logistic loss, which is the ``standard'' loss used in this task, leads to a convex problem, we use square loss to obtain a non-convex objective.}.
Specifically, given training data $\{\x_i,y_i\}_{i=1}^n$, where $\x_i\in\bbR^d, y_i\in\{0,1\}$, consider the following empirical risk minimization problem
\begin{align*}
\min_{\w\bbR^d} \frac1n \sum_{i=1}^n (y_i-\phi(\langle\x_i,\w\rangle))^2,
\end{align*}
where $\phi(z)$ is the sigmoid function, i.e. $\phi(z) = \frac{1}{1+e^{-z}}$.
Datasets are taken from \texttt{LIBSVM} library \cite{libsvm}; see \cref{tab:data1}. 
\begin{table}[htb]
\caption{Datasets for Binary Classification.}
\label{tab:data1}
\centering
\begin{tabular}{cccc}
\toprule
\sc Data & Training Size ($n$) & \# Features ($d$) 
\\
\midrule
{\tt covertype} & $464,810$ & $54$ 
\\
{\tt ijcnn1} & $49,990$ & $22$ 
\\
\bottomrule
\end{tabular}
\end{table}

\begin{figure}[htbp]
	\centering
	\subfigure[Comparison between variants of TR algorithms]{
		\includegraphics[width=.45\textwidth]{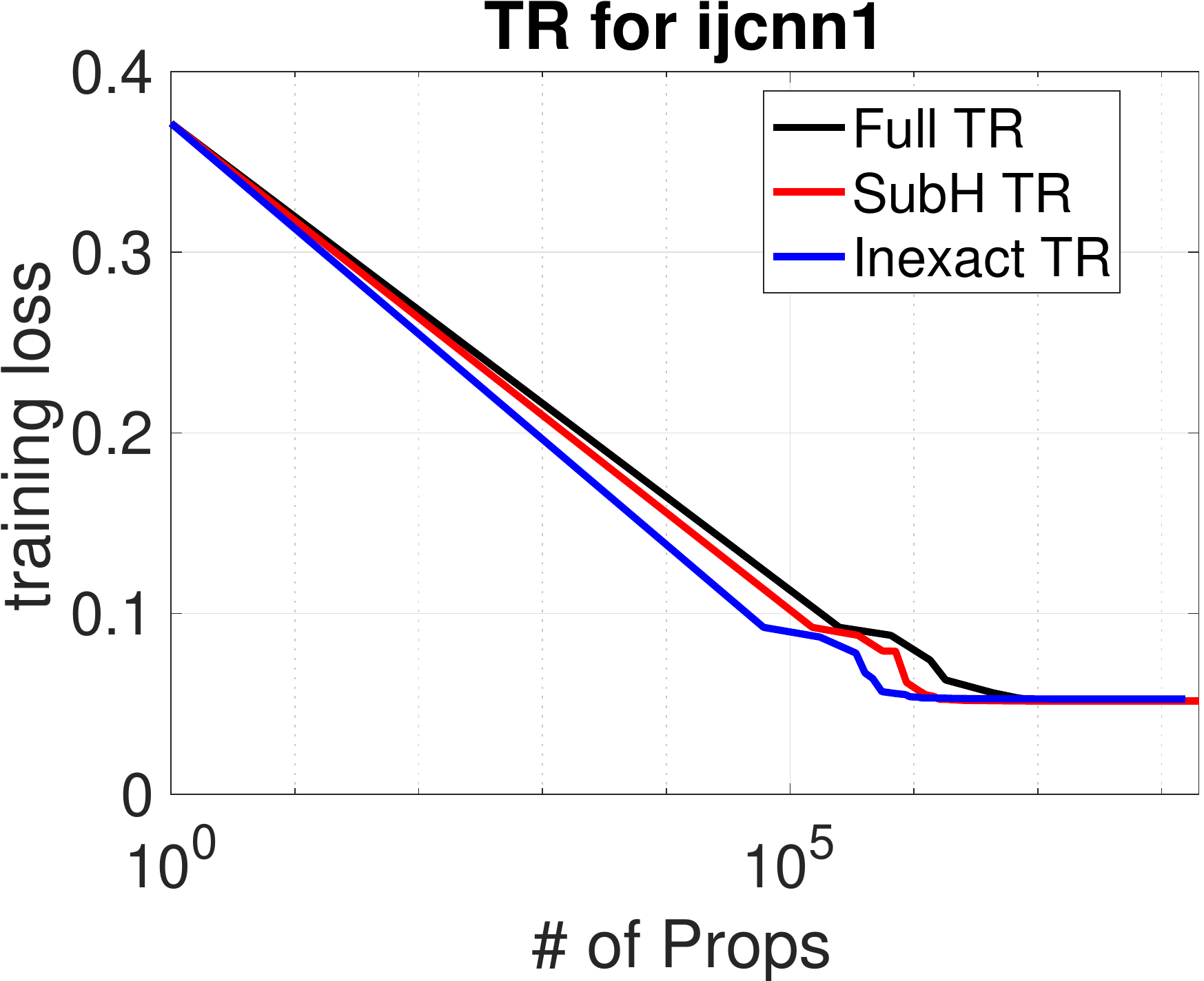}
		\includegraphics[width=.45\textwidth]{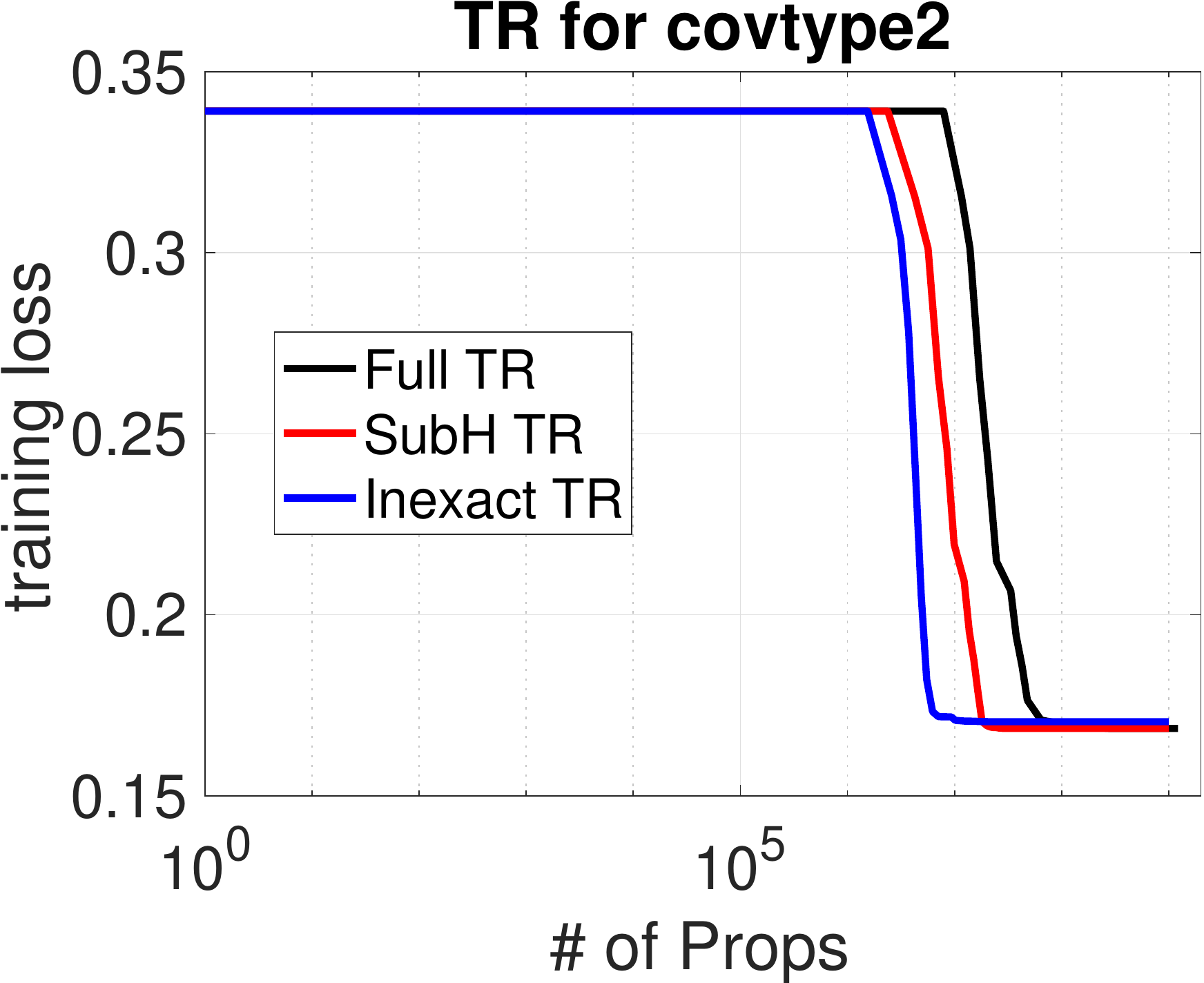}
	}\label{fig:TR_result}
	\subfigure[Comparison between variants of CR algorithms]{
		\includegraphics[width=.45\textwidth]{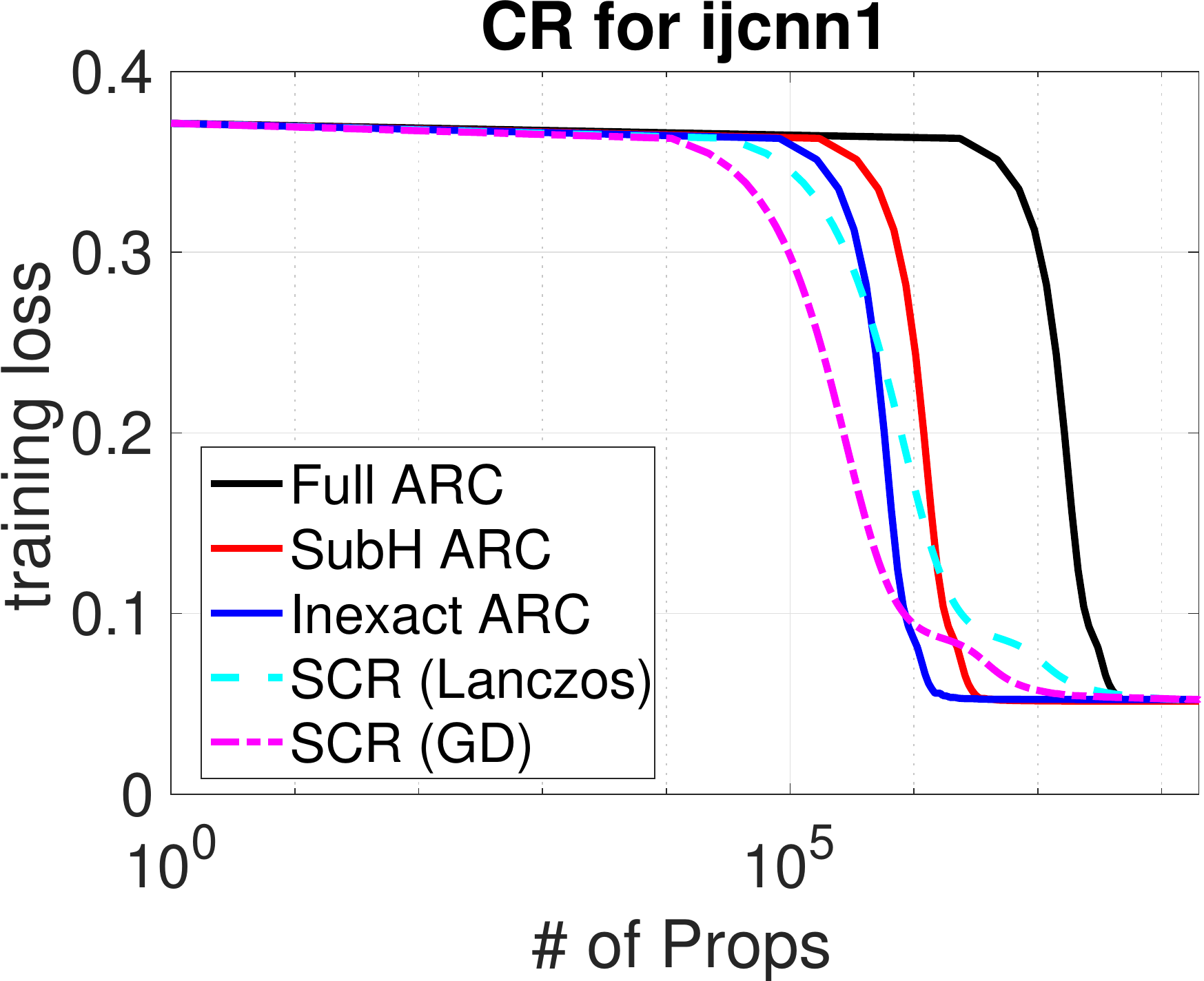}
		\includegraphics[width=.45\textwidth]{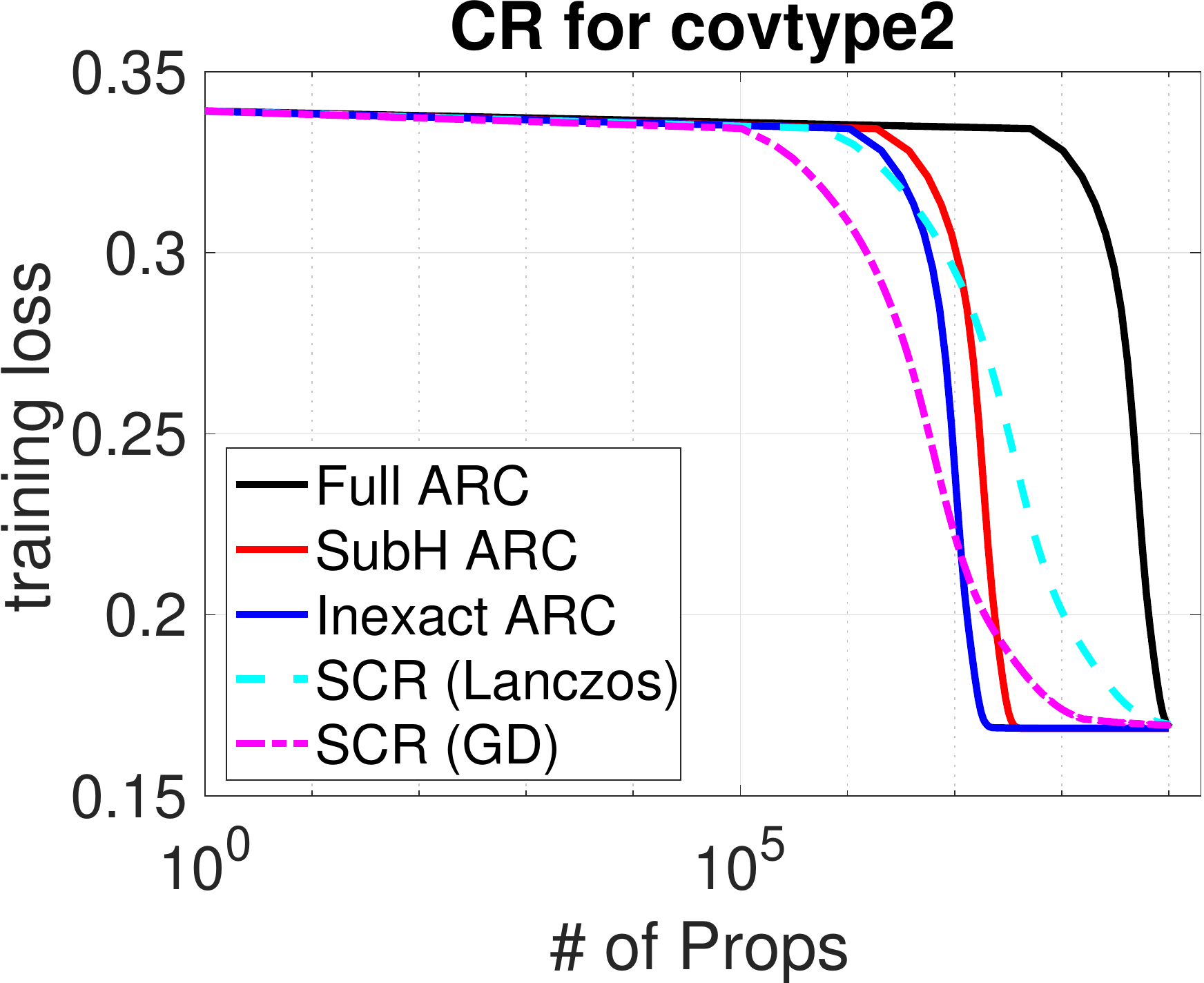}
	}
	\label{fig:CR_result}
	\caption{Performance of various methods on \texttt{ijcnn1} and \texttt{covertype} for binary linear classification. The x-axis is drawn on the logarithmic scale.
	}
	\label{fig:convergent_comparision}
\end{figure}

\begin{figure*}[htbp]
	\begin{center}
		\subfigure[]{
			\includegraphics[width=.3\textwidth]{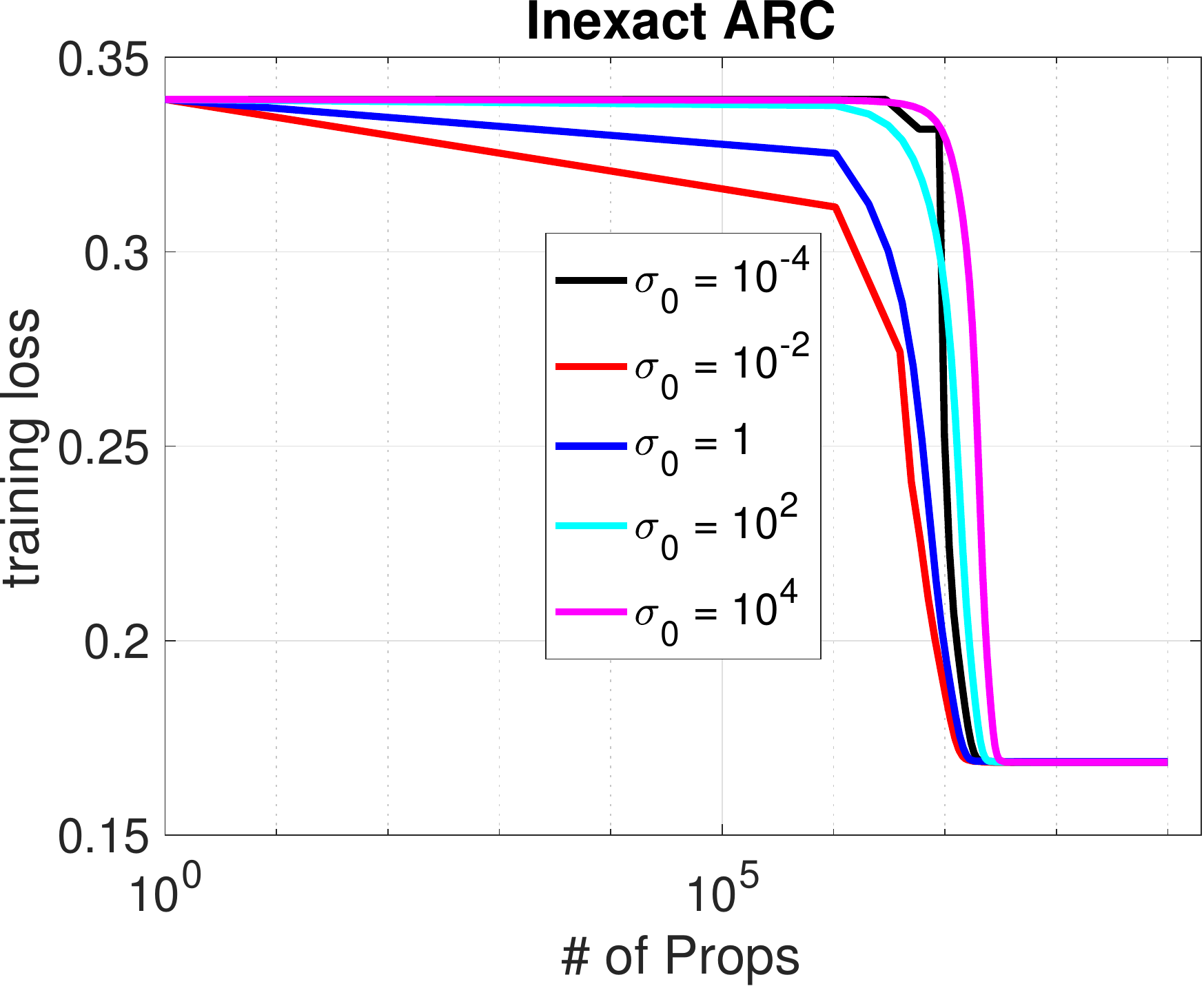}
		}
		\subfigure[]{
			\includegraphics[width=.3\textwidth]{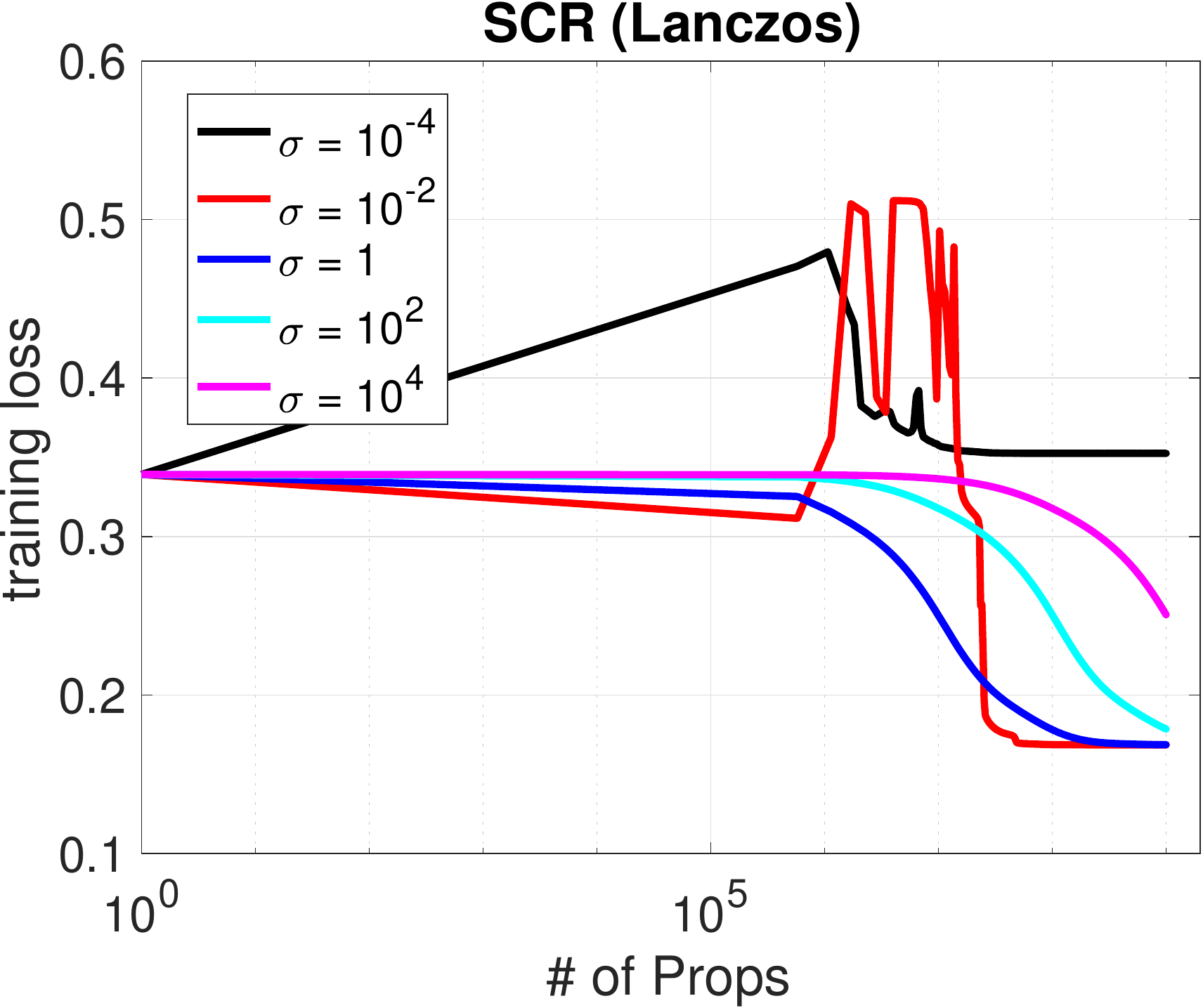}
		}
		\subfigure[]{
			\includegraphics[width=.3\textwidth]{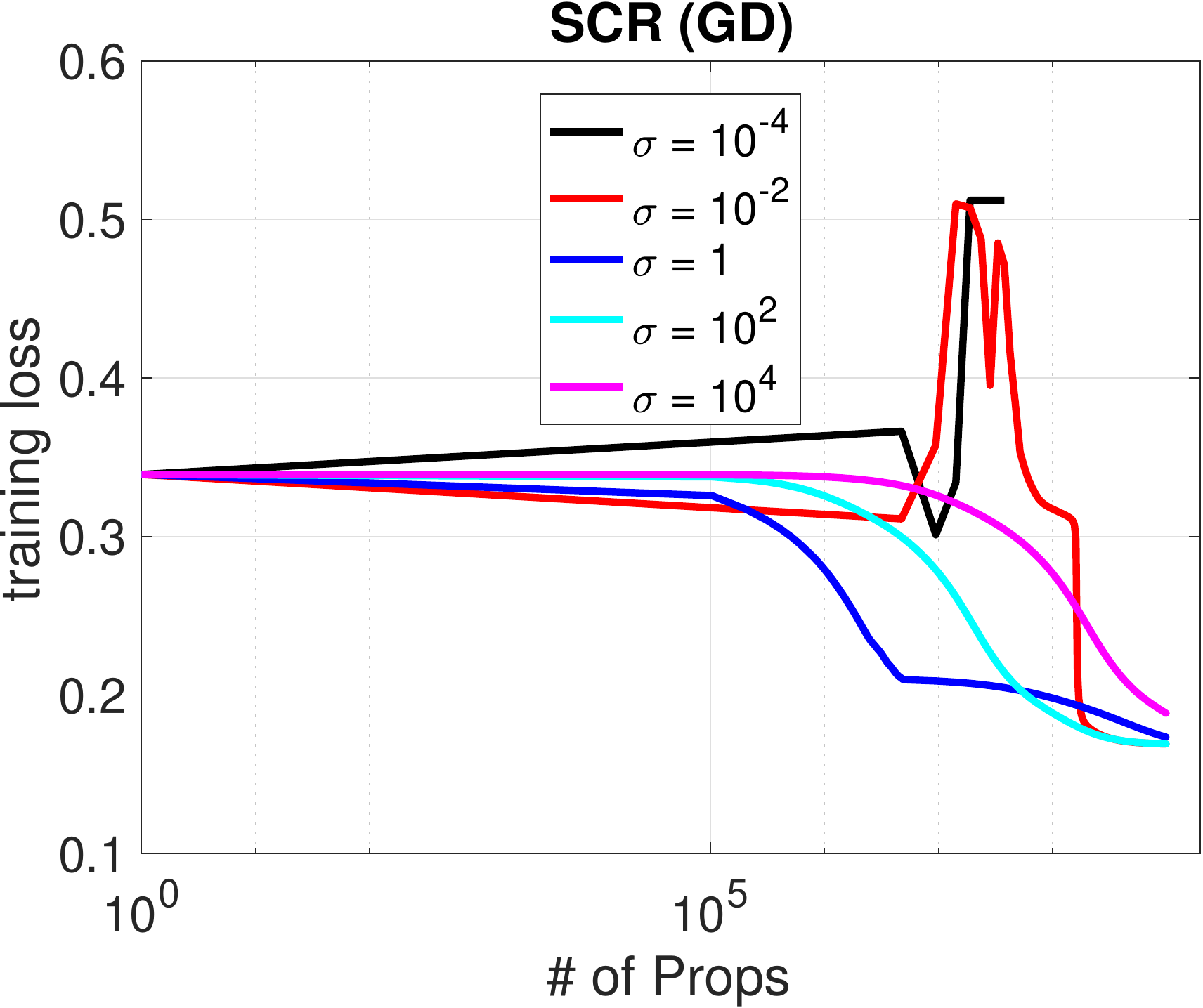}
		}
	\end{center}
	\caption{Robustness of Algorithm~\ref{alg:arc} and sensitivity of SCR w.r.t. the cubic regularization parameter on \texttt{covertype} dataset. For Algorithm~\ref{alg:arc}, this parameter, initially set to $ \sigma_{0} $, adaptively changes across iterations; while for SCR, it is kept fixed at a certain $ \sigma $ for all iterations. (a) Robustness of Algorithm~\ref{alg:arc} to the choice of $\sigma_0$, where $\sigma_0$ varies over several orders of magnitude. (b)--(c) Sensitivity of SCR with two different sub-problem solvers (Lanczos and GD) and several choices of the fixed cubic regularization $\sigma$. For SCR (GD), the step size of GD for solving the sub-problem is hand-tuned to obtain the best performance (which can be extremely expensive).}
	\label{fig:robustness_comparision}
\end{figure*} 

The performance of the following methods are compared:
\begin{itemize}[leftmargin=*,wide=0em]
\item {\it Full TR/ARC}: Standard TR and ARC algorithms with exact gradient and Hessian.
\item {\it SubH TR/ARC} \citep{xuNonconvexTheoretical2017}: TR and ARC with exact gradient and sub-sampled Hessian.
\item {\it SCR (GD)} \citep{tripuraneni2017stochasticcubic}: CR with sub-sampled gradient and Hessian. The sub-problems are solved by gradient descent (GD) \citep{carmon2016gradient}.
\item {\it SCR (Lanczos)}: CR which is similar to SCR (GD) \citep{tripuraneni2017stochasticcubic} but the sub-problems are solved by generalized Lanczos method \citep{cartis2011adaptiveI}.
\item {\it Inexact TR/ARC} ({\bf this work}): TR and ARC with sub-sampled gradient and Hessian as described in Algorithms \ref{alg:trust_region} and \ref{alg:arc}. The sub-problems of Algorithms \ref{alg:trust_region} and \ref{alg:arc} are solved, respectively, by CG-Steihaug \citep{steihaug1983conjugate}, and by generalized Lanczos method \cite{cartis2011adaptiveI}. 
\end{itemize}

Similar to \citet{xuNonconvexEmpirical2017}, the performance of all the algorithms in our experiments is measured by tallying total \textit{number of propagations}, i.e., number of oracle calls of function, gradient and Hessian-vector products. For all TR and ARC algorithms, we use the same setup in \citet{xuNonconvexEmpirical2017}. For all experiments, the gradient and Hessian sampling ratios are $10\%$ and $1\%$ of the entire dataset, respectively.
\vspace{-3mm}
\paragraph{Computational Efficiency (\cref{fig:convergent_comparision}):}
~

First, we compare these Newton-type methods in terms of running time, as measured by the training loss versus total number of propagations; \cref{fig:convergent_comparision} depicts the results. For all variants of SCR, we hand-tuned the algorithm by performing an exhaustive  grid-search over the involving hyper-parameters, and we show the best results. For all variants of TR and ARC, we choose the same initial parameters, i.e. trust region radius for TR and $\sigma_0$ for ARC. 

We can observe that all methods achieve similar training errors, while Algorithms \ref{alg:trust_region} and \ref{alg:arc} do so with much fewer number of propagation calls, as compared with other members of their method class. For example, Inexact TR appears 3-5 times faster than SubH TR and 5-10 times faster than Full TR. Also, all variants of TR perform similarly, or better, than all variants of CR. This is an empirical evidence that the ``optimal'' worst-case analysis of CR, while theoretically interesting, might not translate to many practical applications of interest. 

\vspace{-3mm}
\paragraph{Robustness to Hyper-parameters (\cref{fig:robustness_comparision}):}
~

Next, we highlight the practical challenges arising with algorithms that heavily rely on the knowledge of hard-to-estimate parameters, and how this problem is solved by our methods since our algorithms are formulated so as not to need unknowable problem-related quantities.
In particular, we aim here to demonstrate that an algorithm whose performance is greatly affected by different settings of parameters that cannot be easily estimated, lacks the versatility needed in many practical applications.
To do so, we perform one such demonstration by focusing on sensitivity/robustness of Algorithm~\ref{alg:arc} and SCR \citep{tripuraneni2017stochasticcubic} to the cubic regularization parameter $\sigma$.

Recall that a significant difference between Algorithm~\ref{alg:arc} and SCR is that, unlike the former, the latter requires many hyper-parameter tuning and knowledge of several quantities, e.g., regularization parameter $\sigma$ (which is kept fixed across iterations), Lipschitz constants of gradient and Hessian.
The result is shown in \cref{fig:robustness_comparision}. One can see that the performance of SCR is highly dependent the choice of its main hyper-parameter, i.e., $\sigma$.
Indeed, if $ \sigma $ is not chosen appropriately, SCR either converges very slowly or does not converge at all.
To determine the appropriate value of $ \sigma $ requires an expensive (in human time or CPU time) hyper-parameter search.
This is in sharp contrast with Algorithm \ref{alg:arc} which shows great robustness to the choice of $\sigma_0$ and works more-or-less ``out of the box.''

%% file: conclusion.tex
\section{Conclusions}
In this paper, we considered inexact variants of trust region and adaptive cubic regularization in which, to increase efficiency, the gradient and Hessian, as well as the solution to the underlying sub-problems are all suitably approximated. Our algorithms, and their analysis, do not require knowledge of any unknowable parameter and hence, are easily implementable in practice. We showed that under mild conditions on all these approximation, to coverage to second-order criticality, the inexact variants achieve the same optimal iteration complexity as the exact counterparts. The advantages of our algorithms were also numerically demonstrated.

%% file: appendix.tex
\section{Proofs of the Main Theorems}
Our proof techniques follow similar line of reasoning as in \cite{nesterov2006cubicconstrained,conn2000trust,cartis2011adaptiveI,cartis2011adaptiveII,xuNonconvexTheoretical2017}. However, as alluded to in \cref{remk:proofs}, the mild requirement on the gradient and Hessian approximations as in Condition \ref{ass:approximation} introduces many challenges. 
In the remainder of this section, we give the proof details of main results for Algorithms \ref{alg:trust_region} and \ref{alg:arc}, respectively, in Section \ref{subsec:tr_proof} and \ref{subsec:arc_opt1}-\ref{subsec:arc_proof}. 

%% file: tr_proof.tex
\subsection{Proofs of TR results}\label{subsec:tr_proof}
The proof mainly follows \cite{xuNonconvexTheoretical2017,conn2000trust}. To bound the total iteration numbers, we need to show that the trust region radius never gets too small, i.e. $\Deltat \geq \Delta_{lower}>0$ for all $t$; we do that in \cref{lemma:lowerbound_deltat}.  For that we require some preliminary lemmas. \cref{lemma:cauchy,lemma:eig} gives the sufficient descent obtained with Cauchy and Eigen points.  \cref{lemma:fun} shows the approximation error of $m_t(\s_t)$ as predictor for $F(\x_t+\s_t)-F(\x_t)$. Using these lemmas, we then establish the upper bound on the total number of iterations, as in \cref{lemma:upperbound_success_TR}. 

We now turn to more details. 
The following two lemmas could be found in \cite{conn2000trust}, which establish \cref{cond:tr_approx_sln}.
\begin{lemma}[Cauchy Points]\label{lemma:cauchy}
Suppose that $\s_t^C = \arg\min_{\|\alpha \g_k\| \le \Delta_t} m_t(-\alpha \g_t)$. Then we have
\begin{equation}\label{eqn:cauchy_point_trust}
-m_t(\s_t^C) \geq \frac12 \|\g_t\|\min\{\frac{\|\g_t\|}{1+\|\H_t\|}, \Delta_t\}.
\end{equation}
\end{lemma}
\begin{lemma}[Eigen points]\label{lemma:eig}
When $\lambdam(\H_t)$ is negative, suppose $\u_t$ satisfied 
\begin{equation}\label{ass:eigenpoint_trust}
\<\g_t, \u_t\>\leq 0,~~~~and ~~~~~\<\u_t, \H_t\u_t\>\leq -\nu|\lambdam(\H_t)|\|\u_t\|^2.
\end{equation}
Let $\s_t = \arg\min_{\|\s_t\|\le\Delta_t} m_t(\alpha \u_t)$, we have
\begin{equation}\label{eqn:eigen_points_trust}
-m_t(\s_t) \geq  \frac\nu2|\lambdam(\H_t)|\|\u_t\|^2.
\end{equation}
\end{lemma}

The above two lemmas show the descent that can be obtained by Cauchy and Eigen Points. The following lemma bounds the difference between the actual descent, i.e., $F(\x_t+\s_t)-F(\x_t)$, and the one predicted by $m(\s_t)$.
\begin{lemma}\label{lemma:fun}
Under Assumptions~\ref{ass:regularity}, we have
\begin{equation}\label{eqn:estimation_m_F_trust}
F(\x_t+\s_t) - F(\x_t)-m_t(\x_t) \leq \<\s_t, \nabla F(\x_t) - \g_t\> + \frac12 \delta_H\|\s_t\|^2 + \frac12 L_{F} \|\s_t\|^3.
\end{equation}
\end{lemma}
\begin{proof} Using taylor expansion of $F(\x_t)$ at point $\x_t$, 
\begin{align*}
F(\x_t+\s_t) - F(\x_t)-m_t(\x_t)
&= \<\s_t, \nabla F(\x_t)-\g_t)\> + \frac12 \<\s_t, (\nabla^2 F(\x_t+\tau\s_t)-\H_t) \s_t\>\\
& \leq \<\s_t,\nabla F(\x_t)-\g_t\> + |\frac12 \<\s_t, (\H_t-\nabla^2 \F(\x_t+\tau\s_t)) \s_t\>|\\
& \leq \<\s_t, \nabla F(\x_t)-\g_t\> + |\frac12 \<\s_t, (\H_t-\nabla^2 F(\x_t)) \s_t\>| \\
&~~~~~~~~~~~~~~~~~~~~~+ |\frac12\<\s_t, (\nabla^2 F(\x_t+\tau\s_t)-\nabla^2 F(\x_t)) \s_t\>|\\
&\leq \<\s_t, \nabla F(\x_t)-\g_t\> + \frac12\delta_H\|\s_t\|^2 + \frac12 L_{F} \|\s_t\|^3,
\end{align*}
where $\tau\in[0,1]$.
When $\|\g_t\|>\epsilong$, we can get a loose bound for \eqref{eqn:estimation_m_F_trust} 
\begin{equation}\label{eqn:loose_estimation_m_F_trust}
|F(\x_t+\s_t) - F(\x_t)-m_t(\x_t)| \leq \delta_g\Deltat + \frac12\delta_H \Deltat^2 + \frac12 L_{F} \Deltat^3.
\end{equation}

\end{proof}

By combining \cref{lemma:fun,eqn:loose_estimation_m_F_trust}, \cref{lemma:cauchy_trust} guarantees that, in case $\|\g_t\|\geq\epsilong$, the iteration is successful and the update is accepted.
\begin{restatable}[]{lemma}{lemmacauchytrust}\label{lemma:cauchy_trust}
Given Assumption~\ref{ass:regularity} and~\ref{ass:approximation}, and Condition~\ref{cond:tr_gh} and \ref{cond:tr_approx_sln}, suppose at iteration $t$, $\|\g_t\|\geq\epsilong$ and
$$
\delta_g < \frac{1-\eta}{4}\epsilong, ~\Deltat \leq \min\left\{\frac{\epsilong}{1 + K_H}, \sqrt{\frac{(1-\eta)\epsilong}{12L_H}}, \frac{(1-\eta)\epsilong}{3}\right\},
$$
\noindent then the iteration t is successful, i.e. $\Delta_{t+1} = \gamma\Deltat$.
\end{restatable}

\begin{proof}
First, by Condition \ref{cond:tr_approx_sln}, Lemma \ref{lemma:cauchy} and $\|\g_t\|\geq\epsilong$, we have,
\begin{align*}
-m_t(\s_t) 
&\geq \frac12 \|\g_t\|\min\{\frac{\|\g_t\|}{1+\|\H_t\|}, \Deltat\} \\
& \geq \frac12 \|\g_t\|\min\{\frac{\epsilong}{1+\|\H_t\|}, \Deltat\}\\
&=\frac12 \epsilon_g\Deltat.
\end{align*}
Now according to Lemma \ref{lemma:fun}, we have
\begin{align*}
1 - \rho_t 
&= \frac{F(\x_t+\s_t)-F(\x_t)-m_t(\s_t)}{-m_t(\s_t)} \\
&\leq \frac{\delta_g\Deltat + \frac12 \delta_H\Delta_t^2 + \frac12 L_{F} \Deltat^3}{\frac12 \epsilong\Deltat}\\
& = 2\frac{\delta_g}{\epsilong} + \frac{\delta_h}{\epsilong} \Deltat + \frac{L_{F}}{\epsilong} \Deltat^2\\
& \leq \frac{1-\eta}2 + \frac{\delta_H}{\epsilong} \Deltat + \frac{L_{F}}{\epsilong} \Deltat^2.
\end{align*}
Since $\delta_H<1$, it follows 
$$
\frac{-\delta_H+\sqrt{\delta_H^2+2L_H(1-\eta)\epsilong}}{2L_H} \ge \frac{-1+\sqrt{1+2L_H(1-\eta)\epsilong}}{2L_H}.
$$
Now, we consider two cases. If $2L_H(1-\eta)\epsilong \le 1$, it is not hard to show that
$$
-1+\sqrt{1+2L_H(1-\eta)\epsilong} \ge \frac{2L_H(1-\eta)\epsilong}{3}.
$$
Otherwise, if $2L_H(1-\eta)\epsilong>1$, then it can be shown that
$$
-1+\sqrt{1+2L_H(1-\eta)\epsilong} \ge \sqrt{\frac{L_H(1-\eta)\epsilong}{3}}.
$$
By assumption $\Deltat \le \min\{\sqrt{\frac{(1-\eta)\epsilong}{12L_H}}, \frac{(1-\eta)\epsilong}{3}\}$, it follows,
$$1 - \rho_t \le \frac{1-\eta}2 + \frac{\delta_H}{\epsilong} \Deltat + \frac{L_{F}}{\epsilong} \Deltat^2 \le 1 - \eta,$$
which implies that the iteration $t$ is successful.
\end{proof}

Dealing with the first order term in \cref{eqn:estimation_m_F_trust} is particularly challenging when $\|\g_t\|<\epsilong$. If we simply substitute the result of \cref{lemma:eigen_trust}, i.e. $-m_t(\s_t^E) = \bigO(\lambda_{\min}(\H_t))$ in $1-\rho_t$, it is not hard to see we need to bound a term as $c{\epsilong}/{\epsilonh}$, which indicates $\epsilonh \gg \epsilong$. That is unacceptable. Therefore, after getting Eigen Points $\s_t^E$, we can use either of $\s_t=\s_t^E$ or $\s_t=-\s_t^E$, which gives larger descent. By this simple trick, we could drop $\<\s_t, \nabla F(\x_t)\>$ in our proof; see following lemma for more details.

\begin{restatable}[]{lemma}{lemmaeigentrust}\label{lemma:eigen_trust}
Given Assumption~\ref{ass:regularity} and~\ref{ass:approximation}, and Condition~\ref{cond:tr_gh} and \ref{cond:tr_approx_sln}, suppose at iteration $t$, $\|\g_t\|<\epsilong$ and $\lambdam(H_t) < -\epsilonh$. Then according to (\ref{eqn:subproblem_trust})
$$
m_t(\s) = \frac12 \<\s,\H_t\s\>,
$$
and according to (\ref{eqn:eigen_points_trust}), $\s_t$ satisfies, 
$$
-m_t(\s_t) \geq -m_t(\s_t^E) \geq   \frac\nu2|\lambdam(\H_t)|\Deltat^2.
$$
If $\delta_H < \dfrac{1-\eta}{2}\nu\epsilon_H, \Deltat \leq (1-\eta)\dfrac{\nu\epsilonh}{L_{F}}$,
then the iteration $t$ is successful, i.e. $\Delta_{t+1} = \gamma\Deltat$.
\end{restatable}

\begin{proof} First, review (\ref{eqn:estimation_m_F_trust}),
\begin{align*}
F(\x_t+\s_t) - F(\x_t)-m_t(\x_t) 
&\leq \<\s_t, \nabla F(\x_t)\> + \frac12 \delta_H\|\s_t\|^2 + \frac12 L_{F} \|\s_t\|^3.
\end{align*}
Since either $\s_t$ or $-\s_t$ could be a searching direction, at least one of 
$$
\<\s_t, \nabla F(\x_t)\> \leq 0~~~~~~or~~~~~~\<-\s_t, \nabla F(\x_t)\> \leq 0
$$
is true. W.l.o.g, assume $\<\s_t, \nabla F(\x_t)\>\leq 0$. 
Then 
\begin{align*}
F(\x_t+\s_t) - F(\x_t)-m_t(\x_t) 
&\leq \frac12 \delta_H\|\s_t\|^2 + \frac12 L_{F} \|\s_t\|^3
\end{align*}
Therefore,
\begin{align*}
1 - \rho_t 
&= \frac{F(\x_t+\s_t)-F(\x_t)-m_t(\s_t)}{-m_t(\s_t)} \\
&\leq \frac{\frac12 \delta_H\|\s_t\|^2 + \frac12 L_{F} \|\s_t\|^3}{\frac\nu2|\lambdam(\H_t)|\Deltat^2}\\
&\leq  \frac{\frac12 \delta_H\|\s_t\|^2 + \frac12 L_{F} \|\s_t\|^3}{\frac\nu2\epsilonh\Deltat^2}\\
& \leq  \frac{\frac12 \delta_H\Deltat^2 + \frac12 L_{F} \Deltat^3}{\frac\nu2\epsilonh\Deltat^2}\\
& = \frac{\delta_H}{\nu\epsilon_H} + \frac{L_{F}\Deltat}{\nu\epsilonh} \\
&< (1 - \eta)/2+(1 - \eta)/2\\
&< 1-\eta,
\end{align*}
where the last second inequality uses the condition of $\delta_H$ and $\Deltat$.
Therefore, $\rho_t\geq \eta$ and the iteration is successful.
\end{proof}

Based on \cref{lemma:eigen_trust,lemma:cauchy_trust}, the following lemma helps to get the lower bound of $\Deltat$, whose proof could be found in \citet{xuNonconvexTheoretical2017}.
\begin{lemma}\label{lemma:lowerbound_deltat}
Under Assumption~\ref{ass:regularity} and A.2, Condition C.1, and 
$$
\delta_g < \frac{1-\eta}{4}\epsilong, ~~~ \delta_H <\min\{ \frac{1-\eta}{2}\nu\epsilon_H,1\}.
$$
for Algorithm we have for all t,
$$
\Deltat \geq \frac{1}{\gamma} \min\left\{\frac{\epsilong}{1 + K_H}, \sqrt{\frac{(1-\eta)\epsilong}{12L_H}}, \frac{(1-\eta)\epsilong}{3}, \frac{\nu\epsilonh}{L_{F}}\right\}
$$
\end{lemma}

As a consequence, we now can give the upper bound of successful iterations.
\begin{lemma}[Successful iterations] 
\label{lemma:upperbound_success_TR}	
Given Assumption~\ref{ass:regularity} and~\ref{ass:approximation}, and Condition~\ref{cond:tr_gh} and \ref{cond:tr_approx_sln}, let $\mathscr{T}_\text{succ}$ denote the set of all the successful iterations before Algorithm stops. The the number of successful iterations is upper bounded by
$$
\Abs{\mathscr{T}_\text{succ}} \leq \frac{F(\x_0) - F(\x^*)}{C\epsilonh\min\{\epsilong^2, \epsilonh^2 \}},
$$
where $C$ is a constant depending on $L_{F}, K_H, \delta_g, \delta_H, \eta,\nu$.
\end{lemma}
\begin{proof}
Suppose \cref{alg:trust_region} doesn't terminate at iteration $t$. Then either $\|\g_t\|\geq\epsilong$ or $\lambdam(\H_t)\leq -\epsilonh$.
If $\|\g_t\|\geq\epsilong$, according to (\ref{eqn:cauchy_point_trust}), we have
\begin{align*}
-m_t(\s_t) 
&\geq \frac12 \|\g_t\|\min\{\frac{\|\g_t\|}{1+\|\H_t\|}, \Deltat\}\\
&\geq \frac12 \epsilong\min\{\frac{\epsilong}{1+K_H}, C_0\epsilong, C_1\epsilonh\}\\
&\geq C_2 \epsilong\min\{\epsilong, \epsilonh\}
\end{align*}
Similialy, in the second case $\lambdam(\H_t)\leq -\epsilonh$, from (\ref{eqn:eigen_points_trust}), 
$$
-m_t(\s_t) \geq \frac12\nu\|\lambdam(\H_t)\|\Deltat^2\geq C_3\epsilonh\min\{\epsilong^2, \epsilonh^2\}.
$$
Since $F(\x)$ is monotonically decreasing, we have
\begin{align*}
F(\x_0) - F(\x^*) &\geq \sum_{t=0}^\infty F(\x_t) - F(\x_{t+1})\\
&\geq \sum_{t\in\mathscr{T}_\text{succ}}F(\x_t) - F(\x_{t+1})\\
& \geq \eta \sum_{t\in\mathscr{T}_\text{succ}} C_3\epsilonh \min\{\epsilong^2, \epsilonh^2\}\\
& \geq \Abs{\mathscr{T}_\text{succ}}C_3\epsilonh \min\{\epsilong^2, \epsilonh^2\}.
\end{align*}
Since one of the aboves cases must happen for a successful iteration, it follows,
$$
\Abs{\mathscr{T}_\text{succ}} \leq \frac{F(\x_0) - F(\x^*)}{C_3\epsilonh\min\{\epsilong^2, \epsilonh^2\}}.
$$
\end{proof}

Using the above lemma, the proof of following theorem could be found in \citet{xuNonconvexTheoretical2017}. 
\thmtrmain*

%% file: arc_opt_for_subopt.tex
\subsection{Proof for Inexact ARC}\label{subsec:arc_proof}

In this section, we will prove \cref{thm:arc_main}. 
The goal is to bound the total number of iterations of \cref{alg:arc} before it terminates. First let's denote $\mathscr T_\text{succ}$ as the set of all the successful iteration and $\mathscr T_\text{fail}$ as the set of all the failure iterations. Now we will upper bound the iteration complexity $T:= \Abs{\mathscr T_\text{succ}} + \Abs{\mathscr T_\text{fail}}$. 

First we present the following lemma that gives an upper bound of $\Abs{\mathscr T_\text{fail}}$.

\begin{lemma}\label{lemma:arc_fail}
In \cref{alg:arc}, suppose we have $\sigma_t\le C$, where $C$ is some constant, for all the iteration $t$ before it stops. Then we have $\Abs{\mathscr T_\text{fail}} \le \Abs{\mathscr T_\text{succ}} + \bigO(1)$.
\end{lemma}
\begin{proof}
Since $\sigma_t \le C$, then $\sigma_T = \sigma_0\gamma^{\Abs{\mathscr T_\text{succ}} - \Abs{\mathscr T_\text{fail}}}\le C$. Then immediately we obtain 
$$
\Abs{\mathscr T_\text{fail}} \le \log(C/\sigma_0)/\log\gamma + \Abs{\mathscr T_\text{fail}} = \Abs{\mathscr T_\text{fail}} + \bigO(1).
$$
\end{proof}
Now the remaining analysis is first to show indeed there is a uniform upper bound for all $\sigma_t$ and second to bound number of all the successful iterations.

Following \cite{xuNonconvexTheoretical2017}, we have a similar Lemma ~\ref{lemma:arc_cauchy_lemma} as \citet[Lemma 15]{xuNonconvexTheoretical2017}.

\begin{lemma}[Cauchy Point]\label{lemma:arc_cauchy_lemma}
When $\|\g_t\|\geq\epsilong$, let
$$
\s_t^C = \arg\min_{\alpha\geq0} m_t(-\alpha \g_t).
$$
Then we have 
\begin{subequations}
\begin{align}
\|\s_t^C\| &= \frac{1}{2\sigmat}(\sqrt{K_t^2+4\sigma_t\|\g_t\|} - K_t) . \label{eqn:s_t^cbound}
\\
-m_t(\s_t^C) &\ge\max\left\{\frac{1}{12} \|\s_t^C\|^2 (\sqrt{K_t^2+4\sigma_t\|\g_t\|} - K_t),\frac{\|\g_t\|}{2\sqrt3}\min\{\frac{\|\g_t\|}{|K_t|},\frac{\|\g_t\|}{\sqrt{\sigma_t\|\g_t\|}} \}\right\},
\end{align}
\end{subequations}
where $K_t = \frac{\langle\H_t\g_t,\g_t\rangle}{\|\g_t\|^2}$.
\end{lemma}
\begin{proof}
First, we have
$$
\<\g_t,\s_t^C\> + \<\s_t^C, \H_t\s_t^C\> + \sigma_t\|\s_t^C\|^3=0.
$$
Since $\s_t^C = -\alpha \g_t$ for some $\alpha > 0$, 
$$
-\alpha\|\g_t\|^2 + \alpha^2 \<\g_t, \H_t\g_t\> + \sigma_t\alpha^3\|\g_t\|^3=0.
$$
We can find explicit formula for such $\alpha$ by finding the roots of the quadratic function
$$
r(\alpha) = -\|\g_t\|^2 + \alpha \<\g_t, \H_t\g_t\> + \sigma_t\alpha^2\|\g_t\|^3.
$$
We have
$$
\alpha = \frac{ \<\g_t, \H_t\g_t\> + \sqrt{ \<\g_t, \H_t\g_t\>^2 + 4\sigma_t\|\g_t\|^5}}{2\sigma_t\|\g_t\|^3},
$$
and
$$
2\alpha\sigmat\|\g_t\| = \sqrt{(K_t^2 + 4\sigma_t\|\g_t\|} - K_t.
$$
Hence, it follows that
\begin{equation*} 
\|\s_t^C\| = \alpha \|\g_t\| = \frac{1}{2\sigmat}(\sqrt{K_t^2+4\sigma_t\|\g_t\|} - K_t).
\end{equation*}
Now, from \citet[Lemma 2.1]{cartis2012complexity}, we get
$$
-m_t(\s_t^C) \geq \frac16 \sigma_t\|\s_t^C\|^3 = \frac16\sigmat \|\s_t^C\|^2 \alpha\|\g_t\| =  \frac{1}{12} \|\s_t^C\|^2 (\sqrt{K_t^2+4\sigma_t\|\g_t\|} - K_t).
$$
Alternatively, we have
\begin{align*}
m_t(\s_t^C) &\leq m_t(-\alpha \g_t) = -\alpha\|\g_t\|^2 + \frac12 \alpha^2\<\g_t,\H_t\g_t\> + \frac{\alpha^3}{3}\sigma_t\|\g_t\|^3\\
&= \frac{\alpha\|\g_t\|^2}{6}(-6+3\alpha K_t + 2\alpha^2\sigma_t\|\g_t\|) .
\end{align*}
Consider the quadratic part,
$$
r(\alpha) = -6+3\alpha K_t + 2\alpha^2\sigma_t\|\g_t\|.
$$
We have $r(\alpha)\leq0$ for $\alpha\in[0,\bar \alpha]$, where
$$
\bar\alpha = \frac{-3K_t+\sqrt{9K_t^2+48\sigma_t\|\g_t\|}}{4\sigma_t\|\g_t\|}.
$$
We can express $\bar\alpha$ as
$$
\bar\alpha = \frac{12}{3K_t+\sqrt{9K_t^2+48\sigma_t\|\g_t\|}}.
$$
Note that,
$$
\sqrt{9K_t^2+48\sigma_t\|\g_t\|} \leq 3|K_t| + 4\sqrt{3\sigma_t\|\g_t\|}\leq 8\sqrt3\max\{|K_t|, \sqrt{\sigma_t\|\g_t\|}\}.
$$
Also,
$$
3K_t \leq 2\sqrt3\max\{|K_t|, \sqrt{\sigma_t\|\g_t\|}\} \leq 4\sqrt3\max\{|K_t|, \sqrt{\sigma_t\|\g_t\|}\}.
$$
Hence, defining 
$$
\alpha_0 = \frac{1}{\sqrt3\max\{|K_t|, \sqrt{\sigma_t\|\g_t\|}\}},
$$
it is clear that $0\leq \alpha_0\leq\bar\alpha$. With $\alpha_0$, we have 
$$
r(\alpha_0) \leq 2/3 + 3/\sqrt3 -6 \leq -3
.$$
So finally, we get
$$
m_t(\s_t^C) \leq \frac{-3\|\g_t\|^2}{6\sqrt3}\frac{1}{\max\{|K_t|, \sqrt{\sigma_t\|\g_t\|}\}}
$$
$$
=  \frac{-\|\g_t\|^2}{2\sqrt3}\min\{\frac{1}{|K_t|},\frac{1}{\sqrt{\sigma_t\|\g_t\|}} \}
$$
$$
=  \frac{-\|\g_t\|}{2\sqrt3}\min\{\frac{\|\g_t\|}{|K_t|},\frac{\|\g_t\|}{\sqrt{\sigma_t\|\g_t\|}} \}.
$$
\end{proof}

\begin{lemma}[Eigen Point]\label{lemma:arc_eigen_lemma}
Suppose $\lambdam(\H_t)<0$ and for some $\nu\in(0,1]$, let
$$
\s_t^E =\arg\min_{\alpha\in R}m_t(\alpha \u_t),
$$
where $\u_t$ is the approximate most negative eigenvector defined as
$$
\<\u_t,\H_t\u_t\>\leq \nu\lambdam(\H_t)\|\u_t\|^2\leq 0.
$$
We have 
\begin{subequations}
\begin{align}
\norm{\s_t^E} &\geq \frac{\nu\Abs{\lmin(\H_t)}}{\sigma_t}, \label{eqn:s_t^ebound}
\\
-m_t(\s_t^E) &\geq \frac{\nu|\lmin(\H_t)|}{6}\|\s_t^E\|^2. \label{eqn:eigen_points_cubic}
\end{align}
\end{subequations}
\end{lemma}

\begin{proof}
Again, we know that 
$$
\<\g_t,\s_t^E\> + \<\s_t^E, \H_t\s_t^E\> + \sigma_t\|\s_t^E\|^3=0.
$$
Meanwhile, since $-\s_t$ would keep the last two term as the same value, w.l.o.g, we could assume $\<\g_t,\s_t^E\>\leq 0$, which means
$$
\<\s_t^E, \H_t\s_t^E\> + \sigma_t\|\s_t^E\|^3 \geq 0.
$$
Now, from \citet[Lemma 2.1]{cartis2012complexity}
$$
-m_t(\s_t)\geq \frac16\sigma_t\|\s_t\|^3 \geq -\frac16\<\s_t^E, \H_t\s_t^E\>\geq \frac16 \nu|\lambdam(H_t)|\|\s_t^E\|^2.
$$
It follows that
\begin{equation}
\sigma_t\|\s_t^E\| \geq \nu|\lambdam(\H_t)|,
\end{equation}
which gives
$$
\sigma_t\|\s_t^E\|^3 \geq \frac{\nu^3}{\sigma_t^2}|\lambdam(\H_t)|^3.
$$
\end{proof}


The following lemma gives the bound of the difference between the decrease of the objective function and value of the quadratic model $m(\s_t)$.
\begin{lemma} \label{lemma:obj_dg1}
Under \cref{ass:approximation}, we have
\begin{equation}\label{eqn:estimation_m_F_cubic}
F(\x_t+\s_t) - F(\x_t)-m_t(\s_t) \leq  \<\s_t, \nabla F(\x_t)-\g_t\> + \frac12 \delta_H\|\s_t\|^2 + (\frac{L_{F}}{2}-\frac{\sigma_t}{3}) \|\s_t\|^3.
\end{equation}

\begin{proof} Using Taylor expansion of $F(\x)$ at point $\x_t$, 
\begin{align*}
F(\x_t+\s_t) - F(\x_t)-m_t(\x_t) 
&= \<\s_t, \nabla F(\x_t)-\g_t\> + \frac12 \<\s_t, (\nabla^2 F(\x_t + \tau\s_t)-\H_t) \s_t\> - \frac{\sigma_t}{3}\|\s_t\|^3\\
& \leq \<\s_t, \nabla F(\x_t)-\g_t\> + |\frac12 \<\s_t, (\H_t-\nabla^2 F(\xii_t)) \s_t\>|- \frac{\sigma_t}{3}\|\s_t\|^3\\
& \leq \<\s_t, \nabla F(\x_t)-\g_t\> + |\frac12 \<\s_t, (\H_t-\nabla^2 F(\x_t)) \s_t\>| \\
&+ |\frac12 \<\s_t, (\nabla^2 F(\xii_t)-\nabla^2 F(\x_t)) \s_t\>|- \frac{\sigma_t}{3}\|\s_t\|^3\\
&\leq \<\s_t,\nabla F(\x_t)-\g_t\> + \frac12 \delta_H\|\s_t\|^2 + (\frac{L_{F}}{2}-\frac{\sigma_t}{3}) \|\s_t\|^3,
\end{align*}
where $\tau\in [0,1]$.
\end{proof}
\end{lemma}




Based on the above lemmas, the following lemma shows that iteration $t$ is successful when $\|\g_t\|\ge\epsilong$.

\begin{lemma} \label{lemma:cauchy_arc_succ} Given Condition~\ref{cond:arc_sub_appr}, when $\|\g_t\|\ge \epsilong$, $\sigmat \ge 2L_F$ with
$$
\delta_g \leq \frac{1-\eta}{12}\epsilong~~~~~\text{and}~~~~\delta_H \leq \frac{1-\eta}{6}\sqrt{2\epsilong L_H},
$$
then the iteration $t$ is successful, i.e. $\sigma_{t+1} = \sigmat/\gamma$.
\end{lemma}

\begin{proof} From \cref{eqn:estimation_m_F_cubic}, we could get
\begin{align*}
F(\x_t+\s_t^C) - F(\x_t)-m_t(\s_t^C) 
&\leq  \delta_g\|\s_t^C\|+ \frac12 \delta_H\|\s_t^C\|^2 + (\frac{L_{F}}{2}-\frac{\sigma_t}{3}) \|\s_t^C\|^3\\
&\leq \delta_g\|\s_t^C\|+ \frac12 \delta_H\|\s_t^C\|^2,
\end{align*}
since $\sigmat \geq 2L_F$. We divide it to two cases.

First, if $K_t = \frac{\langle\H_t\g_t,\g_t\rangle}{\|\g_t\|^2}\leq 0$, then from \cref{eqn:s_t^cbound}, it follows 
\[
    \|\s_t^C\| \geq \frac{1}{2\sigmat}\sqrt{4\sigmat\|\g_t\|} = \sqrt{\|\g_t\|/\sigmat}.
\]
Using the result in \citet[Lemma 2.1]{cartis2012complexity}, we get
\begin{align*}
1 - \rho_t 
&= \frac{F(\x_t+\s_t)-F(\x_t)- m_t(\s_t)}{-m_t(\s_t)} \\
& \leq \frac{\delta_g\norm{\s_t^C} + \frac12\delta_H\norm{\s_t^C}^2}{\frac{\sigmat\|\s_t^C\|^3}{6}} \\
& = \frac{\delta_g + \frac12\delta_H\norm{\s_t^C}}{\frac{\sigmat\|\s_t^C\|^2}{6}} \\
& \le \frac{6\delta_g}{\epsilong} + \frac{3\delta_H}{\sqrt{2\epsilong L_H}}\\
& \le \frac{1-\eta}2 +\frac{1-\eta}2 = 1-\eta,
\end{align*}
where the last inequality follows from the condition on $\delta_g$ and $\delta_H$.

For the second case where $K_t > 0$, it follows that
\[
    \|\s_t^C\| = \frac{\sqrt{K_t^2+4\sigmat\|\g_t\|}-K_t}{2\sigmat} = \frac{2\|\g_t\|}{\sqrt{K_t^2+4\sigmat\|\g_t\|}+K_t}.
\]
Now we consider two cases: {\bf{(a)}} $ K_t^2 \geq \sigmat\|\g_t\|$ and {\bf(b)} $ K_t^2 \leq \sigmat\|\g_t\|$.

{\bf(a)} When $K_H^2 \geq K_t^2 \geq \sigmat\|\g_t\|$, from above equality we have 
\[
    \|\s_t^C\| \leq \frac{\|\g_t\|}{K_t}.
\]
Meanwhile, since $K_t^2 \geq \sigmat\|\g_t\|$, from \cref{lemma:arc_cauchy_lemma}, we have 
\[
    -m_t(\s_t^C) \geq \frac{\|\g_t\|}{2\sqrt3}\min\{\frac{\|\g_t\|}{|K_t|},\frac{\|\g_t\|}{\sqrt{\sigma_t\|\g_t\|}} \} = \frac{\|\g_t\|^2}{2\sqrt3 K_t}.
\]
Combine above inequality together, it follows
\begin{align*}
1 - \rho_t 
&= \frac{F(\x_t+\s_t)-F(\x_t)- m_t(\s_t)}{-m_t(\s_t)} \\
& \leq \frac{\delta_g\norm{\s_t^C} + \frac12 \delta_H\norm{\s_t^C}^2}{\frac{\|\g_t\|^2}{2\sqrt3 K_t}} \\
& \leq \frac{\delta_g\frac{\|\g_t\|}{K_t} + \frac12 \delta_H (\frac{\|\g_t\|}{K_t})^2} {\frac{\|\g_t\|^2}{2\sqrt3 K_t}} \\
& = \frac{2\sqrt3\delta_g}{\|\g_t\|} + \frac{\sqrt3 \delta_H}{K_t}\\
& \leq \frac{2\sqrt3\delta_g}{\epsilong} + \frac{\sqrt3 \delta_H}{\sqrt{2L_F\epsilong}}\\
& \le \frac{1-\eta}2 +\frac{1-\eta}2 = 1-\eta.
\end{align*}

{\bf(b)} When $K_t^2 \leq \sigmat\|\g_t\|$, we have
\[
    \|\s_t^C\| \leq \frac{\|\g_t\|}{\sqrt{\|\g_t\sigma_t\|}},
\]
and  
\[
    -m_t(\s_t^C) \geq \frac{\|\g_t\|}{2\sqrt3}\min\{\frac{\|\g_t\|}{|K_t|},\frac{\|\g_t\|}{\sqrt{\sigma_t\|\g_t\|}} \} \geq \frac{\|\g_t\|^{3/2}}{2\sqrt3 \sqrt{\sigmat}}.
\]
Then, 
\begin{align*}
1 - \rho_t 
&= \frac{F(\x_t+\s_t)-F(\x_t)- m_t(\s_t)}{-m_t(\s_t)} \\
& \leq \frac{\delta_g\norm{\s_t^C} + \frac12 \delta_H\norm{\s_t^C}^2}{\frac{\|\g_t\|^{3/2}}{2\sqrt3 \sqrt{\sigmat}}} \\
& = \frac{2\sqrt3\delta_g}{\|\g_t\|} + \frac{\sqrt3 \delta_H}{\sqrt{\sigmat\epsilong}}\\
& \leq \frac{2\sqrt3\delta_g}{\epsilong} + \frac{\sqrt3 \delta_H}{\sqrt{2L_F\epsilong}}\\
& \le \frac{1-\eta}2 +\frac{1-\eta}2 = 1-\eta.
\end{align*}

From the above, we could see that iteration $t$ is successful, i.e. $\sigma_{t+1}=\sigmat/\gamma$, when $\|\g_t\|\geq\epsilong$.

\end{proof}

The following lemma, whose proof can be found in \citet[Lemma 17]{xuNonconvexTheoretical2017}, helps bound $F(\x_t+\s_t) - F(\x_t)-m_t(\x_t) $ when the Hessian has negative eigenvalues. 
\begin{lemma}\label{lemma:m_t_bound_by_s_t^E_cubic}
Given \cref{ass:regularity} and \cref{ass:approximation} suppose
$$
\sigma_t\ge 2L_{F},~~
\delta_H \leq \frac{\nu}{6}\epsilon_H.
$$
Then, we have
$$
\frac12 \delta_H\|\s_t\|^2 + (\frac12 L_{F}-\frac{\sigma_t}{3}) \|\s_t\|^3 \leq \frac{\delta_H}{2} \|\s_t^E\|^2~~~~~\text{if $\lmin(\H_t)<-\epsilonh$}.
$$

\end{lemma}


Then, the following lemma shows Eigen Points also yields a descent by using the same trick as \cref{lemma:eigen_trust}.
\begin{restatable}[]{lemma}{lemmaeigenarc}\label{lemma:eigen_arc}
Given \cref{ass:regularity},~\ref{ass:approximation} and Condition~\ref{cond:arc_sub_appr},\ref{cond:arc_opt_gh}. suppose at iteration $t$, $\lambdam(\H_t)< -\epsilonh$ and $\|\g_t\|\leq\epsilong$.
Under the assumption $\lambdam(\H_t)< -\epsilonh$ and $\|\g_t\|\leq\epsilong$, recall that our sub-problem is now
$$
m_t(\s) = \frac12\<\s, \H_t\s\> + \frac{\sigma_t}{3}\|\s\|^3.
$$
Then it is clear that if $\s_t$ is a approximating solution of the above problem, so is $-\s_t$. If
$$
\sigmat \geq 2 L_{F},~~
\delta_H \leq \min\left\{\frac{\nu(1-\eta)\epsilonh}{3}, \frac{\nu\epsilonh}{6}, \frac{1-\eta}6\sqrt{2\epsilong L_H}\right\},
$$
then iteration t is successful, i.e. $\sigma_{t+1}=\sigmat/\gamma$.
\end{restatable}\begin{proof}

Since either $\s_t$ or $-\s_t$ is a approximating solution, at least one of 
$$
\<\s_t, \nabla F(\x_t)\> \leq 0~~~~~~or~~~~~~\<-\s_t, \nabla F(\x_t)\> \leq 0
$$
is true. W.l.o.g, assume $\<\s_t, \nabla F(\x_t)\>\leq0$. Then according to (\ref{eqn:estimation_m_F_cubic})
$$
F(\x_t+\s_t) - F(\x_t)-m_t(\x_t) \leq \frac12 \delta_H\|\s_t\|^2 + \frac12 (L_{F}-\frac{\sigma_t}{3}) \|\s_t\|^3.
$$
Therefore, according to (\ref{eqn:eigen_points_cubic}) and Lemma \ref{lemma:m_t_bound_by_s_t^E_cubic},
\begin{align*}
1-\rho_t 
&= \frac{F(\x_t+\s_t)-F(\x_t)- m_t(\s_t)}{-m_t(\s_t)}\\
&\leq \frac{\frac12 \delta_H\|\s_t\|^2 + \frac12 (L_{F}-\frac{\sigma_t}{3}) \|\s_t\|^3}{-m_t(\s_t)}\\
&\leq \frac{\frac{\delta_H}2\|\s_t^E\|^2 }{ \frac{\nu|\lambdam(\H_t)|}{6}\|\s_t^E\|^2}\\
& = \frac{3\delta_H}{\nu\|\lambdam(\H_t)\|}\\
& \leq \frac{3\delta_H}{\nu\epsilon_H}\\
&\leq 1 - \eta,
\end{align*}
which means the iteration $t$ is successful.
\end{proof}

With the help of the above lemmas, we can now show an upper bound for $\sigma_t$, as in Lemma \ref{lemma:upperbound_sigmat}.
\begin{lemma}\label{lemma:upperbound_sigmat}
given  \cref{ass:regularity},~\ref{ass:approximation} and Condition~\ref{cond:arc_sub_appr},\ref{cond:arc_opt_gh}, suppose
\begin{align*}
\delta_H &\leq \min\left\{\frac{\nu(1-\eta)\epsilonh}{3}, \frac{\nu\epsilonh}{6}, \frac{1-\eta}6\sqrt{2\epsilong L_H}\right\},
\\
\delta_g & \le \frac{1 - \eta}{12} \epsilong.
\end{align*}
Then for all t,
$$
\sigmat \leq 2\gamma L_{F}.
$$
\end{lemma}
\begin{proof}
If $\sigma_0 \le 2\gamma L_{F}$,  
we prove by contradiction. Suppose the iteration $t$ is the first unsuccessful iteration such that 
$$\sigma_{t+1} = \gamma\sigma_t \ge 2\gamma L_{F},$$
which implies that 
$$
\sigma_t \ge 2 L_{F}.
$$
However, according to \cref{lemma:cauchy_arc_succ} and \cref{lemma:eigen_arc}, respectively, if $\norm{\g_t} \ge\epsilon_g$ or $\lmin(\H_t) \le -\epsilon_H$, then the iteration is successful and then $\sigma_{t+1} = \sigma_t /\gamma \le \sigma_t$, which is a contradiction. 

If $\sigma_0 > 2 \gamma L_{F}$, since any iteration $t$ with $\sigma_t \ge 2L_{F}$ is successful, then $\sigma_t < \sigma_0$ for some $t$. 
\end{proof}

Now we upper bound the number of all successful iterations $\Abs{\mathscr T_\text{succ}}$, which is shown in \cref{lemma:arc_succ}. The proof is similar to \citet[Lemma 21]{xuNonconvexTheoretical2017}.
\begin{lemma}[Successful iterations]\label{lemma:arc_succ}  \cref{ass:regularity},~\ref{ass:approximation} and Condition~\ref{cond:arc_sub_appr},\ref{cond:arc_opt_gh}, the the number of successful iterations is upper bounded by,
$$
\Abs{\mathscr{T}_\text{succ}} \leq \frac{F(x_0) - F(x^*)}{C} \max\{\epsilong^{-2},\epsilonh^{-3}\}.
$$
\end{lemma}

Based on the above lemmas, it follows,
\thmarcmain*
\begin{proof}
It follows from \cref{lemma:arc_fail} and \cref{lemma:arc_succ}.
\end{proof}

%% file: arc_opt.tex
\subsection{Proof of Optimal Complexity of ARC}\label{subsec:arc_opt1}
For the optimal complexity of ARC, we need more accurate solutions of the subproblem~\cref{eqn:subproblem_cubic} other than just using Cauchy Point when the gradient is not small. Therefore, we need to change Condition~\ref{cond:arc_sub_appr} to \cref{cond:arc_optimal2} . Consequently, we need to refine some lemmas in \cref{subsec:arc_proof}. First, we need use the following result which gives conditions for a successful iteration when $\|\g_t\|\geq \epsilong$.

\begin{lemma} \label{lemma:obj_dg2}
Given \cref{ass:regularity} and \cref{ass:approximation}, $\sigma_t\geq 2 L_{F}$, if
\begin{align*}
\delta_H &\leq \frac{1}{24} (\sqrt{K_H^2+8L_{F}\epsilong}-K_H),
\\
\delta_g &\le \frac{(\sqrt{K_H^2 + 8L_{F} \epsilon_g} - K_H)^2}{192L_{F}},
\end{align*}
then we have
\begin{equation}\label{eqn:s_t_boundby_s_t^C}
\delta_g \|\s_t\|+\frac12 \delta_H\|\s_t\|^2 + (\frac12 L_{F}-\frac{\sigma_t}{3}) \|\s_t\|^3 \leq \delta_g\|\s_t^C\| +\frac{1}{2}\delta_H \|\s_t^C\|^2, ~~~\text{if $\|\g_t\|>\epsilong$}
\end{equation}
\end{lemma}

\begin{proof}
We consider the following two cases:
\begin{enumerate}

\item If $\|\s_t\|\leq \|\s_t^C\|$, then from the assumption of $\sigma_t$, it immediately follows that 
$$
\delta_g \|\s_t\|+\frac12 \delta_H\|\s_t\|^2 + (\frac12 L_{F}-\frac{\sigma_t}{3}) \|\s_t\|^3 \leq \delta_g \|\s_t\|+ \frac12\delta_H\|\s_t\|^2\leq \delta_g \|\s_t^C\|+\frac12\delta_H\|\s_t^C\|^2.
$$

\item If $\norm{\s_t} \ge \norm{\s_t^C}$, first, since $L_{F} \le \sigma_t/2$,
\begin{align*}
\delta_g \|\s_t\|+\frac12 \delta_H\|\s_t\|^2 + (\frac12 L_{F}-\frac{\sigma_t}{3}) \|\s_t\|^3 & \le \delta_g \|\s_t\|+\frac12 \delta_H\|\s_t\|^2 -\frac{\sigma_t}{12} \|\s_t\|^3.
\end{align*}
Now let's define function $r(x) = \delta_g  + \frac12 \delta_H x - \frac{\sigma_t}{12}x^2$. Compute the derivative of $r(x)$ and we obtain
\begin{align*}
r'(x) &= \frac12\delta_H  - \frac{1}{6} \sigma_t x.
\end{align*}
For any $x \ge \norm{\s_t^C}$, according to \cref{eqn:s_t^cbound}, we have
\begin{align*}
r'(x) & \le \frac12 \delta_H  - \frac{1}{6} \sigma_t \norm{\s_t^C} 
\\
& \le \frac12\delta_H - \frac{\sqrt{K_H^2 + 4\sigma_t\epsilon_g} -K_H}{12}
\\
& \le 0.
\end{align*}
Therefore,
\begin{align*}
r(\norm{\s_t}) & \le r(\norm{\s_t^C}) = \delta_g + \frac12\delta_H \norm{\s_t^C} - \frac{1}{12} \sigma_t\norm{\s_t^C}^2
\\
& \le \delta_g + (\frac12\delta_H - \frac{\sqrt{K_H^2 + 4\sigma_t\epsilon_g}-K_H}{24})\norm{\s_t^C}
\\
& \le \delta_g -\frac{\sqrt{K_H^2 + 4\sigma_t\epsilon_g}-K_H}{48}\norm{\s_t^C}
\\
& \le \frac{\sqrt{K_H^2 + 8L_{F}\epsilon_g}-K_H}{192L_{F}} - \frac{\sqrt{K_H^2 + 4\sigma_t\epsilon_g}-K_H}{96\sigma_t}
\\
&\le 0
\end{align*}
The last inequality follows from the fact that function $p(x):= \frac{(\sqrt{a^2 + x} -a)^2}{x}$ is a increasing function over $\bbR_+$.
Then, we have
$$\delta_g \|\s_t\|+\frac12 \delta_H\|\s_t\|^2 + (\frac12 L_{F}-\frac{\sigma_t}{3}) \|\s_t\|^3 = \norm{\s_t} r(\norm{\s_t}) \le 0. $$
This completes the proof.

\end{enumerate}
\end{proof}

With the help of the above lemma, we show that iteration $t$ is succeessful when $\|\g_t\|\geq\epsilong$.
 \begin{lemma} \label{lemma:cauchy_arc_en}
 Given Assumption \ref{ass:regularity}, \ref{ass:approximation}, Condition \ref{cond:arc_opt_gh}, \ref{cond:arc_optimal2}, suppose at iteration t, $\|\g_t\| > \epsilong
 $, $ \sigma_t \ge 2L_{F}$ and
\begin{align*}
 \delta_H &\leq \frac{1-\eta}{24} (\sqrt{K_H^2+8L_{F}\epsilong}-K_H),
 \\
 \delta_g & \le \frac{1 - \eta}{192 L_{F}} (\sqrt{K_H^2 + 8L_{F}\epsilong} - K_H)^2.
\end{align*}
Then, the iteration $t$ is successful, i.e. $\sigma_{t+1}=\sigmat/\gamma$.
 \end{lemma}

\begin{proof}
First, since $\norm{\g_t} \ge \epsilong$, by \cref{lemma:obj_dg1} and \cref{lemma:obj_dg2}, we have
\begin{align*}
F(\x_t + \s_t) - F(\x_t) - m_t(\s_t) \le \delta_g\norm{\s_t^C} + \frac12\epsilon_H\norm{\s_t^C}^2.
\end{align*}
Now from Condition \ref{cond:arc_optimal2} and \cref{eqn:s_t^cbound}, we get
$$
-m_t(\s_t) \geq -m_t(\s_t^C)\geq \frac{1}{12}\|\s_t^C\|^2( \sqrt{K_H^2+4\sigma_t\|\g_t\|} - K_H ).
$$

Consider the approximation quality $\rho_t$, 
\begin{align*}
1 - \rho_t & = \frac{F(\x_t+\s_t)-F(\x_t)- m_t(\s_t)}{-m_t(\s_t)}
\\
& \le \frac{\delta_g\norm{\s_t^C} + \frac12\delta_H\norm{\s_t^C}^2}{\frac{1}{12}\|\s_t^C\|^2( \sqrt{K_H^2+4\sigma_t\|\g_t\|} - K_H )}
\\
& = \frac{12\delta_g}{\|\s_t^C\|( \sqrt{K_H^2+4\sigma_t\|\g_t\|} - K_H )} + \frac{6\delta_H}{ \sqrt{K_H^2+4\sigma_t\|\g_t\|} - K_H }
\\
& \le \frac{24\sigma_t\delta_g}{( \sqrt{K_H^2+4\sigma_t\|\g_t\|} - K_H )^2} + \frac{6\delta_H}{ \sqrt{K_H^2+4\sigma_t\|\g_t\|} - K_H }
\\
& \le \frac{24\sigma_t\delta_g}{( \sqrt{K_H^2+4\sigma_t\epsilon_g} - K_H )^2} + \frac{6\delta_H}{ \sqrt{K_H^2+4\sigma_t\epsilon_g} - K_H }
\\
& \le \frac{48 L_{F}\delta_g}{( \sqrt{K_H^2+8L_{F}\epsilong} - K_H )^2} + \frac{6\delta_H}{ \sqrt{K_H^2+8L_{F} \epsilong} - K_H }
\end{align*}
where the second inequality follows from \cref{eqn:eigen_points_trust} and the last inequality follows from $\sigma_t \ge 2 L_{F}$ as well as the fact that function $r(x):= \frac{x}{(\sqrt{a^2 + x} - a)^2}$ is a monotonically decreasing function over $\bbR_+$.

Since $\delta_H \leq \frac{1-\eta}{24} (\sqrt{K_H^2+4L_{F}\epsilong}-K_H)$, we get 
$\frac{6\delta_H}{ \sqrt{K_H^2+8L_{F} \epsilong} - K_H }
 \le \frac{1 - \eta}{4}$.
 
Since $\delta_g  \le \frac{1 - \eta}{192 L_{F}} (\sqrt{K_H^2 + 8L_{F}\epsilong} - K_H)^2$, we get 
$\frac{48 L_{F}\delta_g}{( \sqrt{K_H^2+8L_{F}\epsilong} - K_H )^2} \le \frac{1-\eta}{4}$.

Therefore,
$1 - \rho_t \le 1 -\eta$, which means the iteration is successful.

\end{proof}

Then, as \cref{lemma:upperbound_sigmat}, we have 
\begin{lemma}\label{lemma:upperbound_sigmat_opt}
Given Assumption \ref{ass:regularity}, \ref{ass:approximation}, Condition \ref{cond:arc_opt_gh}, \ref{cond:arc_optimal2}, suppose
\begin{align*}
\delta_H &\leq \min\left\{\frac{1-\eta}{24} (\sqrt{K_H^2+8L_{F}\epsilong}-K_H), \frac{1 -\eta}{6}\nu\epsilon_H\right\},
\\
\delta_g & \le \frac{1 - \eta}{192 L_{F}} (\sqrt{K_H^2 + 8L_{F}\epsilong} - K_H)^2,
\end{align*}
then $\sigmat\leq2\gamma L_F$ for all $t$.
\end{lemma}

After the above preparation, we can now prove the optimal complexity of \cref{alg:arc} under \cref{cond:arc_optimal2}. Recall that \cref{lemma:upperbound_sigmat} still holds. So we only need to prove a tighter bound for $\Abs{\mathcal T_\text{succ}}$. In particular, we separate $\mathscr T_\text{succ}$ into the following three subsets:
\begin{align}
\mathscr T_\text{succ}^1 &\triangleq \{t\in\mathscr T_\text{succ} \mid \norm{\g_{t+1}} \ge\epsilon_g\}\\
\mathscr T_\text{succ}^2 &\triangleq \{t\in\mathscr T_\text{succ}\mid \norm{\g_{t+1}} \le\epsilon_g \text{ and } \lmin(\H_{t+1}) \le -\epsilon_H\}\\
\mathscr T_\text{succ}^3 &\triangleq \{t\in\mathscr T_\text{succ} \mid \norm{\g_{t+1}} \le\epsilon_g \text{ and } \lmin(\H_{t+1}) \ge -\epsilon_H\}
\end{align}
Clearly, $\mathscr T_\text{succ} = \mathscr T_\text{succ}^1\bigcup \mathscr T_\text{succ}^2 \bigcup \mathscr T_\text{succ}^3$, and,  trivially, $\Abs{\mathscr T_\text{succ}^3} = 1$. 

First, let us bound $\mathscr T_\text{succ}^2$.
\begin{lemma}\label{lemma:arc_opt_T2} 
Given Assumption \ref{ass:regularity}, \ref{ass:approximation}, Condition \ref{cond:arc_opt_gh}, \ref{cond:arc_optimal2}, we have the following upper bound,
$$
    \Abs{\mathscr{T}_\text{succ}^2} \leq C \epsilonh^{-3}.
$$
\end{lemma}
\begin{proof}
Since $F(\x_t)$ is monotonically decreasing, then
\begin{align*}
F(\x_0) - F_{\min} &\ge \sum_{t=0}^{T-1} F(\x_t) - F(\x_{t+1}) = F(\x_0) - F(\x_1) + \sum_{t=0}^{T-1} F(\x_t) - F(\x_{t+1})
\\ &\ge F(\x_0) - F(\x_1) + \sum_{t\in\mathscr T_\text{succ}^2} F(\x_t) - F(\x_{t+1})
\\ & \ge F(\x_0) - F(\x_1) + \sum_{t\in\mathscr T_\text{succ}^2} \eta m_{t+1}(\s_{t+1})
\\ & \ge F(\x_0) - F(\x_1) + \eta\sum_{t\in\mathscr T_\text{succ}^2} \frac{\nu^3\epsilon_H^3}{24\gamma^2L_{F}^2}
\end{align*}
where the last inequality follows from \cref{eqn:eigen_points_cubic}. Hence,
$$\Abs{\mathscr T_\text{succ}^2} \le \frac{(F(\x_1)- F_{\min})24\gamma^2L_{F}^2}{\eta\nu^3} \epsilon_H^{-3} = \bigO(\epsilon_H^{-3}).$$
\end{proof}

Intuitively, we could see that we need each update to yield sufficient descent in order to bound $\mathscr T_\text{succ}^1$. Equivalently, we need each $\s_t$ to be bounded below to get sufficient decrease; see the following lemma.
\begin{lemma}\label{lemma:arc_opt2} 
When iteration  $t$ is successful and $\|\g_t\|\geq \epsilong$, given Assumption \ref{ass:regularity}, \ref{ass:approximation}, Condition \ref{cond:arc_opt_gh}, \ref{cond:arc_optimal2},  we have 
$$
    \|\s_t\| \geq \kappa_g [(1-\zeta-\frac{\zeta}{1-2\zeta})\|\g_{t+1}\| - \dfrac{5}{2}\delta_g],
$$
where 
$$
    \kappa_g = \min\left\{\frac{1}{(\frac{L_{F}}2 + 2\gamma L_{F} + \epsilon_0+\zeta K_F)}, \frac{1}{(\frac{L_{F}}2 + 2\gamma L_{F} + \frac{\zeta}{1-2\zeta}K_F + \zeta K_F)}\right\}.
$$

\end{lemma}
\begin{proof} Using Condition \ref{cond:arc_optimal2}, we get 
\begin{align}\label{eq:gt+1}
    \|\g_{t+1}\| \leq \|\g_{t+1}-\nabla m_t(\s_t)\| + \|\nabla m_t(\s_t)\| \leq \|\g_{t+1}-\nabla m_t(\s_t)\| + \theta_t\|\g_t\|
\end{align}
Noting that $\nabla m_t(\s_t)=\g_t + \H_t\s_t + \sigmat \|\s_t\|\s_t$, Condition \ref{ass:approximation} and \cref{ass:regularity}, we have
\begin{align}
\|\g_{t+1} - \nabla m_t(\s_t)\| 
&\leq \|\g_{t+1}-\g_t - \H_t\s_t\| + \sigmat\|\s_t\|^2 \nonumber \\
&\leq \|\int_0^1 (\nabla^2 F(\x_t+\tau\s_t)-\nabla^2 F(\x_t))\s_t d\tau + (\nabla^2 F(\x_t) - \H_t)\s_t\| \nonumber \\
&+ \|\g_t - \nabla F(\x_t)\| + \|\g_{t+1}-\nabla F(\x_t+\tau\s_t)\| + \sigmat\|\s_t\|^2 \nonumber \\
&\leq (\frac{L_{F}}{2} +2\gamma L_{F})\|\s_t\|^2 + \delta_H\|\s_t\| + 2 \delta_g.
\label{eq:gt+12}
\end{align}
Also according to Condition \ref{ass:approximation}, we get
\begin{align}
\|\g_{t}\| 
&\leq \|\g_{t} - \nabla F(\x_t)\| + \|\nabla F(\x_t))\|  \nonumber \\
& \leq \delta_g + K_H\|\s_t\| + \|\nabla F(\x_t+\s_t)\| \nonumber \\
&\leq 2\delta_g + K_H\|\s_t\| + \|\g_{t+1}\|.
\label{eq:gt}
\end{align}
By combining \cref{eq:gt+1,eq:gt+12,eq:gt}, we get
\begin{align*}
    \|\g_{t+1}\| &\leq (\frac{L_{F}}2 + 2\gamma L_{F})\|\s_t\|^2 + (\delta_H +\theta_t K_F) \|\s_t\| + 2 (1+\theta_t)\delta_g + \theta_t\|\g_{t+1}\|
    \\
    & \le (\frac{L_{F}}2 + 2\gamma L_{F})\|\s_t\|^2 + (\delta_H +\theta_t K_F) \|\s_t\| + \dfrac52\delta_g + \zeta\|\g_{t+1}\|,
\end{align*}
which implies 
$$
    (1-\zeta)\|\g_{t+1}\| - \dfrac52\delta_g \leq (\frac{L_{F}}2 + 2\gamma L_{F})\|\s_t\|^2 + (\delta_H +\theta_t  K_F) \|\s_t\| .
$$
Now, consider two cases:
\begin{enumerate}

\item If $\|\s_t\|\geq 1$, then
$$
(\delta_H +\theta_t K_F) \|\s_t\| \leq (\epsilon_H +\zeta K_F) \|\s_t\|^2.
$$
It follows,
$$
    (1-\zeta)\|\g_{t+1}\| - 5/2\delta_g \leq (\frac{L_{F}}2 + 2\gamma L_{F} + \epsilon_H+\zeta K_F)\|\s_t\|^2.
$$
i.e.
$$
\norm{\s_t^2} \ge \frac{(1-\zeta)\|\g_{t+1}\| - \dfrac52\delta_g}{\frac{L_{F}}2 + 2\gamma L_{F} + \epsilon_H+\zeta K_F}.
$$

\item If $\|\s_t\|\leq 1$, then
\begin{align*}
\delta_H  &\leq \zeta\|\g_t\|
\\ &\le \zeta(\norm{\g_{t+1}} + \norm{\nabla F(\x_t + \s_t) - \g_{t+1}} + \norm{\nabla F(\x_t) - \nabla F(\x_t + \s_t)} + \norm{\g_t - \nabla F(\x_t)})
\\ &\leq \zeta (2\delta_g + K_F\|\s_t\| + \|\g_{t+1}\|)
\\ &\leq \zeta (2\delta_H + K_F\norm{\s_t} + \norm{\g_{t+1}})
\end{align*}
where the last inequality follows from $\delta_g \le \delta_H$ in \cref{eq:arc_opt2} in \cref{cond:arc_optimal2}.
Therefore we have
$$
\delta_H \|\s_t\| \leq \frac{\zeta}{1-2\zeta} ( K_F\|\s_t\| + \|\g_{t+1}\|) \|\s_t\|\leq \frac{\zeta}{1-2\zeta} ( K_F\|\s_t\|^2 + \|\g_{t+1}\|).
$$
Then,
$$
    (\delta_H +\theta_t K_F) \|\s_t\| \leq (\frac{\zeta}{1-2\zeta} + \zeta)K_F \|\s_t\|^2 + \frac{\zeta}{1-2\zeta}\|\g_{t+1}\|.
$$
That implies 
$$
    (1-\zeta-\frac{\zeta}{1-2\zeta})\|\g_{t+1}\| - \dfrac52\delta_g \leq (\frac{L_{F}}2 + 2\gamma L_{F} + \frac{\zeta}{1-2\zeta}K_F + \zeta K_F) \|\s_t\|^2,
$$
i.e.
$$
\norm{\s_t}^2 \ge \frac{(1-\zeta-\frac{\zeta}{1-2\zeta})\|\g_{t+1}\| - \dfrac52\delta_g}{\frac{L_{F}}2 + 2\gamma L_{F} + \frac{\zeta}{1-2\zeta}K_F + \zeta K_F}.
$$
\end{enumerate}
The two cases complete the proof.
\end{proof}

Now, based on \cref{lemma:arc_opt2}, it is not hard to bound $\Abs{\mathscr{T}_\text{succ}^1}$.
\begin{lemma}\label{lemma:arc_opt_T1}
Given the same setting as \cref{lemma:arc_opt2}, then the success iterations $\mathscr{T}_\text{succ}^1$ based on $\|\g_t\|\geq \epsilong$ is bounded by
$$
    \Abs{\mathscr{T}_\text{succ}^1} \leq C\max\{\epsilong^{-1.5}, \epsilonh^{-3}\}.
$$
\end{lemma}

\begin{proof}
First, according to \cref{eq:arc_g} in \cref{cond:arc_gh}, we have
$$
\delta_g \le \frac{1 - \eta}{192 L_{F}} (\sqrt{K_H^2 + 8L_{F}\epsilong} - K_H)^2 \le \frac{1 - \eta}{192 L_{F}} 8L_{F}\epsilon_g \le \frac{1}{24}\epsilong.
$$
If $\norm{\g_{t+1}}\geq \epsilong$, according to \cref{lemma:arc_opt2}, we have 
\begin{align*}
\|\s_t\|^2 \geq \kappa_g [(1-1/4-\frac{1/4}{1-2/4})\epsilong - 5/2\frac{1}{24}\epsilong] = \frac1{8} \kappa_g \epsilong.
\end{align*} 

Now consider any $t\in\mathscr T_\text{succ}^1$. If $\norm{\g_t} \ge \epsilon_g$, then we have
\begin{align*}
-m_t(\s_t) \ge \frac{\sigma_t}{6}\norm{\s_t}^3 \ge \frac{\sigma_{\min}}{6}(\frac{\kappa_g\epsilon_g}{8})^{3/2} \ge c_g\epsilon^{3/2},
\end{align*}
where $c_g\triangleq \frac{\kappa_g^{3/2}\sigma_{\min}}{200}$.
Otherwise, we must have $\lmin(\H_t)\le -\epsilon_H$, and by \cref{eqn:eigen_points_cubic}, we have
\begin{align*}
-m_t(\s_t) \ge \frac{\nu^3\epsilon_H^3}{24\gamma^2 L_{F}^2} \le c_H\epsilon_H^3,
\end{align*}
where $c_H \triangleq \frac{\nu^3}{24\gamma^2L_{F}^2}$.
Therefore,
\begin{align*}
-m_t(\s_t) \ge \min\{c_g\epsilon_g^{3/2}, c_H\epsilon_H^3\}.
\end{align*}
Since $F(\x_t)$ is monotonically decreasing and $F(\x)$ is lower bounded by $F_{\min}$, then
\begin{align*}
F(\x_0) - F_{\min} &\ge \sum_{t=0}^{T-1} F(\x_t) -F(\x_{t+1}) 
\\ & \ge \sum_{t\in\mathscr T_\text{succ}^1} F(\x_t) -F(\x_{t+1}) 
\\ & \ge \sum_{t\in\mathscr T_\text{succ}^1} -\eta m_t(\s_t)
\\ & \ge \sum_{t\in\mathscr T_\text{succ}^1} \min\{c_g\epsilon_g^{3/2}, c_H\epsilon_H^3\}
\\ & = \Abs{\mathscr T_\text{succ}^1} \min\{c_g\epsilon_g^{3/2}, c_H\epsilon_H^3\}.
\end{align*}
Therefore $$
\Abs{\mathscr T_\text{succ}^1} \le \max\left\{\frac{F(\x_0) - F_{\min}}{c_g}\epsilon_g^{-3/2}, \frac{F(\x_0) - F_{\min}}{c_H}\epsilon_H^{-3}\right\},
$$
which completes the proof.

\end{proof}

 Since $\mathscr T_\text{succ} = \mathscr T_\text{succ}^1\bigcup \mathscr T_\text{succ}^2 \bigcup \mathscr T_\text{succ}^3$, we can get the bound of total number of successful iterations. 
\begin{lemma}Given Assumption \ref{ass:regularity}, \ref{ass:approximation}, Condition \ref{cond:arc_opt_gh}, \ref{cond:arc_optimal2}, then the success iterations $\mathscr{T}_\text{succ}$ is bounded by 
$$
        \Abs{\mathscr{T}_\text{succ}} \leq C\max\{\epsilong^{-1.5}, \epsilonh^{-3}\}.
$$
\end{lemma}
\begin{proof}
It immediately follows from \cref{lemma:arc_opt_T2}.
and \cref{lemma:arc_opt_T1}.
\end{proof}

Using the same technique as \cref{thm:arc_main}, we could prove:
\thmarcmainoptII*

\textbf{Remark.} If we assume $L_{F}$ is known (set $\sigma_t \equiv L_{F}$) and $\s_t$ is close enough to the best solution $\s_t^*$ of $m_t(\s)$, by using Taylor expansion, it is not hard to show that
\[
    F(\x_t+\s_t) - F(\x_t) \geq -c_1 m_t(\s_t) \geq -c_2 m_t(\s_t^*).
\]
Given $\|\g_t\|$ or $-\lambda_{\min}(\H_t)$ is large, $-m(\s_t^*)$ would then be large. Therefore, there could be enough descent along $\s_t$. Roughly speaking, we could drop \cref{lemma:arc_cauchy_lemma} to \ref{lemma:upperbound_sigmat}, and get the same iteration complexity results, i.e. $T \in \mathcal{O}(\max\{\epsilong^{-1.5},\epsilonh^{-3}\}$. For example, we do not need \cref{lemma:arc_cauchy_lemma} to show Cauchy Point is one of the directions for $-m_t(\s_t)$. Also, either \cref{lemma:obj_dg2} or \cref{lemma:cauchy_arc_en} is redundant.